\documentclass[a4paper]{amsart}
\usepackage{amsthm,amsfonts,amsmath,amssymb}
\usepackage[abs]{overpic}
\usepackage{thm-restate}

\newtheorem{theorem}{Theorem}[section]
\newtheorem{proposition}[theorem]{Proposition}

\newtheorem{claim}[theorem]{Claim}

\newtheorem{example}[theorem]{Example}
\newtheorem{conjecture}[theorem]{Conjecture}
\newtheorem{remark}[theorem]{Remark}

\begin{document}
\title[DELTA-UNKNOTTING NUMBER FOR MONTESINOS KNOTS]{DELTA-UNKNOTTING NUMBER FOR MONTESINOS KNOTS}
\author{Kazumichi Nakamura}
\email{mi-ka-07130703@ezweb.ne.jp}





\begin{abstract}
The $\Delta$-unknotting number for a knot is defined as the minimum number of $\Delta$-moves needed to deform the knot into the trivial knot.
In this paper, we determine the $\Delta$-unknotting numbers for certain families of Montesinos knots. 
Using results on pretzel knots and two-bridge knots, we prove that for these families the $\Delta$-unknotting number equals the absolute value of the second coefficient of the Conway polynomial. 
In particular, every positive pretzel knot belongs to these families, and certain two-bridge knots also belong to them.
\end{abstract}

\maketitle

\noindent\textbf{keywords:}$\Delta$-move, $\Delta$-unknotting number, $\Delta$-Gordian distance, Conway polynomial, Conway's normal form, two-bridge knot, pretzel knot, positive pretzel knot, Montesinos knot, linking number.

\section{Introduction}
In this paper, we study the $\Delta$-unknotting numbers for Montesinos knots.

In \cite{MaN}, H. Murakami and Y. Nakanishi introduced a local move on regular diagrams of oriented knots and links, called a $\Delta$-move (or $\Delta$-unknotting operation), as illustrated in Figure \ref{fig:delta}.
\begin{figure}[htbp]
 \centering    \includegraphics[width=0.8\linewidth]
    {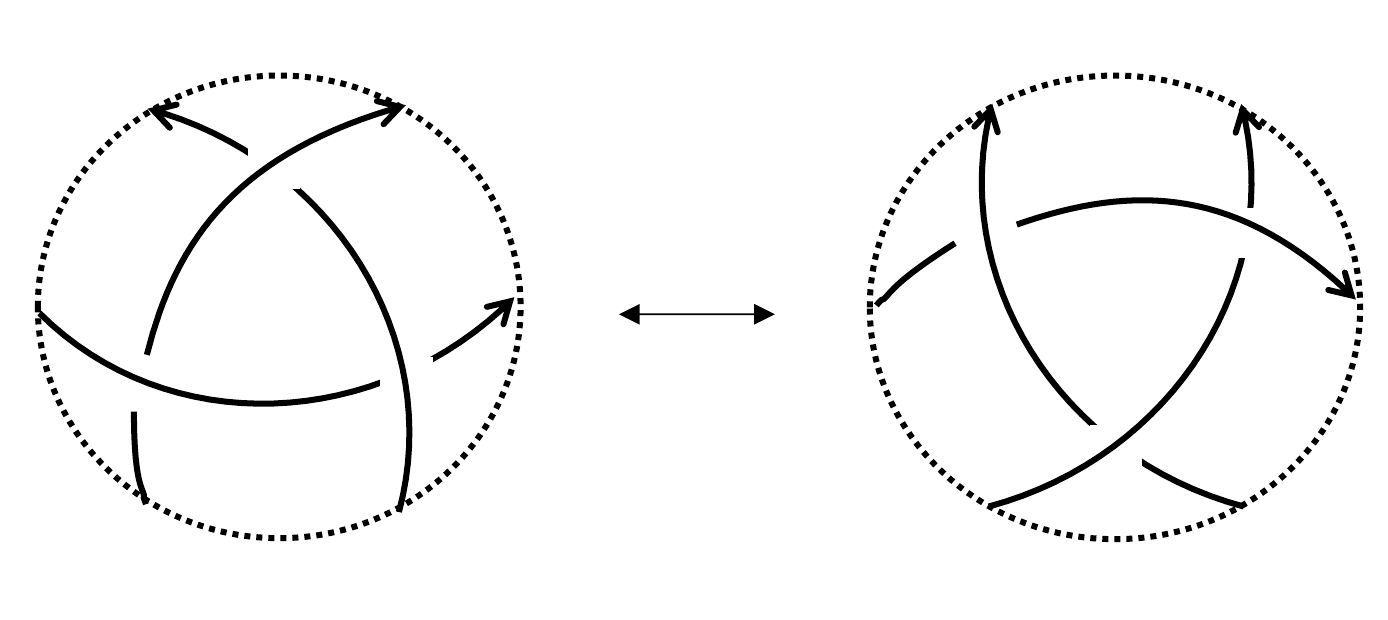}
 \caption{A $\Delta$-move.}
 \label{fig:delta}
\end{figure}
They proved that two knots can be deformed into each other by a finite sequence of $\Delta$-moves. 
The $\Delta$-Gordian distance $d_G^{\Delta}(K, K^{'})$ of two oriented knots $K$ and $K^{'}$ is defined as the minimum number of $\Delta$-moves needed to deform a diagram of $K$ into that of $K{'}$. 
The $\Delta$-unknotting number $u^{\Delta}(K)$  of an oriented knot $K$ is defined as the $\Delta$-Gordian distance $d_G^{\Delta}(K, O)$ of $K$ and the trivial knot $O$.
In \cite{Oka1}, M. Okada proved that $u^{\Delta}(K) \geq |a_2(K)|$, where $a_2(K)$ denotes the second coefficient of the Conway polynomial of $K$. 

In \cite{NaNaU}, joint work with Y. Nakanishi and Y. Uchida, we determined the $\Delta$-unknotting numbers for certain classes of oriented knots: torus knots, positive pretzel knots, and positive 3-braids.
For these classes, the $\Delta$-unknotting number coincides with the second coefficient of the Conway polynomial. 
In particular, the $\Delta$-unknotting number for $T(p,q)$ (the torus knot of type $(p,q)$) is equal to $(p^2-1)(q^2-1)/24$.
In \cite{Naka4}, we computed the $\Delta$-unknotting numbers for positive pretzel knots. 

Furthermore, in \cite{Naka3}, we determined the $\Delta$-unknotting numbers for two-bridge knots of type 
$C(2m_1, 2m_2, ... ,  2m_n)$ and type $C(2m_1, 2m_2, ... , 2m_{n-1}, 2m_n-1)$, where $m_i$ is a positive integer for $1 \leq i \leq n$.
In particular, these numbers coincide with the absolute values of the second coefficients of their Conway polynomials.

\vspace{1em}
In this paper, we study certain families of Montesinos knots and determine their $\Delta$-unknotting numbers. Using results from \cite{Naka4, Naka3}, we show that for these families $u^{\Delta}(K) = |a_2(K)|$ (see Theorem~\ref{thm:main}). In particular, every positive pretzel knot belongs to these families, and every two-bridge knot of type $C(2m_1, 2m_2, ... ,  2m_n)$ or $C(2m_1, 2m_2, ... , 2m_{n-1}, 2m_n-1)$ in
\cite{Naka3} also belongs to them.

We now state our main theorem. 
The definitions of Type~A, B, C, D, E, and F rational tangles 
will be given in Section~2, 
and the notation $V^{(k)}, W^{(k)}, X^{(k)}, Y^{(k)}$ 
will be introduced in Section~3.

\begin{restatable}{theorem}{main}\label{thm:main}
Let
$K = M\big(0; (\alpha_1,\beta_1), \dots, (\alpha_r,\beta_r)\big)$
be a Montesinos knot with $r (\geq 2)$ rational tangles $T(\beta_k/\alpha_k)$ for $1 \leq k \leq r$.
Suppose that one of the following holds:

\begin{enumerate}
\item All tangles $T(\beta_k/\alpha_k)$ are of Type~D, and $r$ is odd.

\item $T(\beta_1/\alpha_1)$ is of Type~E, 
      $T(\beta_k/\alpha_k)$ is of Type~C
      for $2 \le k \le r$, 
      and $r$ is even.

\item $T(\beta_1/\alpha_1)$ is of Type~F, 
      $T(\beta_k/\alpha_k)$ is of Type~C
      for $2 \le k \le r$, 
      and $r$ is odd.

\item $T(\beta_1/\alpha_1)$ is of Type~E, 
      $T(\beta_k/\alpha_k)$ is of Type~A
      for $2 \le k \le r$.

\item $T(\beta_1/\alpha_1)$ is of Type~D, 
      $T(\beta_k/\alpha_k)$ is of Type~A
      for $2 \le k \le r$.
      
\item $T(\beta_1/\alpha_1)$ is of Type~F, 
      $T(\beta_k/\alpha_k)$ is of Type~B
      for $2 \le k \le r$.      
\end{enumerate}

Then, we have $u^{\Delta}(K) = |a_2(K)|.$
\begin{enumerate}
    \item If $K$ is of pattern (1), then
    \begin{align*}
    u^{\Delta}(K) 
= -\sum_{k=1}^{r} W^{(k)} 
+ \sum_{1 \leq i < j}^{r} X^{(i)} X^{(j)} + \frac{1}{8} ( r-1 ) 
\quad \big(= a_2(K)\big).
    \end{align*}
    \item If $K$ is of pattern (2), then
    \begin{align*}
    u^{\Delta}(K)   
& = \sum_{k=1}^{r} W^{(k)} 
+ \frac{1}{2} \sum_{k=1}^{r} {X^{(k)}}^2
+ X^{(1)} \sum_{k=2}^{r} X^{(k)} 
- \frac{1}{8} (r-1) \\
& \big(= a_2(K)\big).
    \end{align*}
    \item If $K$ is of pattern (3), then
    \begin{align*}
    u^{\Delta}(K)   
& = -W^{(1)} + \sum_{k=2}^{r} W^{(k)} 
+ \frac{1}{2} \sum_{k=2}^{r} {X^{(k)}}^2
- X^{(1)} \sum_{k=2}^{r} X^{(k)} 
- \frac{1}{8} (r-1) \\ 
& \big(= a_2(K)\big).
    \end{align*}
    \item If $K$ is of pattern (4), then
    \begin{align*}
    u^{\Delta}(K)  
= W^{(1)} + \sum_{k=2}^{r} V^{(k)} 
+ X^{(1)} \sum_{k=2}^{r} Y^{(k)} 
\quad \big(= - a_2(K)\big).
    \end{align*}
    \item If $K$ is of pattern (5), then
    \begin{align*}
    u^{\Delta}(K)  
= - W^{(1)} + \sum_{k=2}^{r} V^{(k)} 
+ \frac{1}{2} {X^{(1)}}^2
+ X^{(1)} \sum_{k=2}^{r}Y^{(k)}
\quad \big(= a_2(K)\big).
    \end{align*}    
    \item If $K$ is of pattern (6), then
    \begin{align*}
    u^{\Delta}(K)
= - W^{(1)} + \sum_{k=2}^{r} V^{(k)}
+ \frac{1}{2} \sum_{k=2}^{r} {{Y^{(k)}}^2}
+ X^{(1)} \sum_{k=2}^{r} Y^{(k)}
\quad \big(= a_2(K)\big).
    \end{align*} 
\end{enumerate}
\end{restatable}

In Section \ref{sec5}, we consider Montesinos knots with the same minimal crossing number that admit minimal crossing diagrams sharing the same shadow, where this shadow realizes the minimal crossing number.

\vspace{1em}
\section{Preliminaries}\label{sec2}

\subsection{Two-Bridge Knots and Their $\Delta$-Unknotting Numbers} 
\label{subsec21}

\leavevmode

A knot $K$ is said to be a two-bridge knot if $K$ has a diagram as in Figure \ref{fig:twobridge}, called Conway's normal form. For a knot diagram as in Figure \ref{fig:twobridge}, each $|c_i|$ presents the number of half-twists for integers $c_1, c_2, ... , c_n$. In this paper, for the sign of $\alpha_i$, we assume that a right-handed half-twist is positive when $i$ is odd, and a left-handed half-twist is positive when $i$ is even. We denote this knot diagram by $C(c_1, c_2, ... , c_n)$.
\begin{figure}[htbp]
 \centering
 \includegraphics[width=1\linewidth]{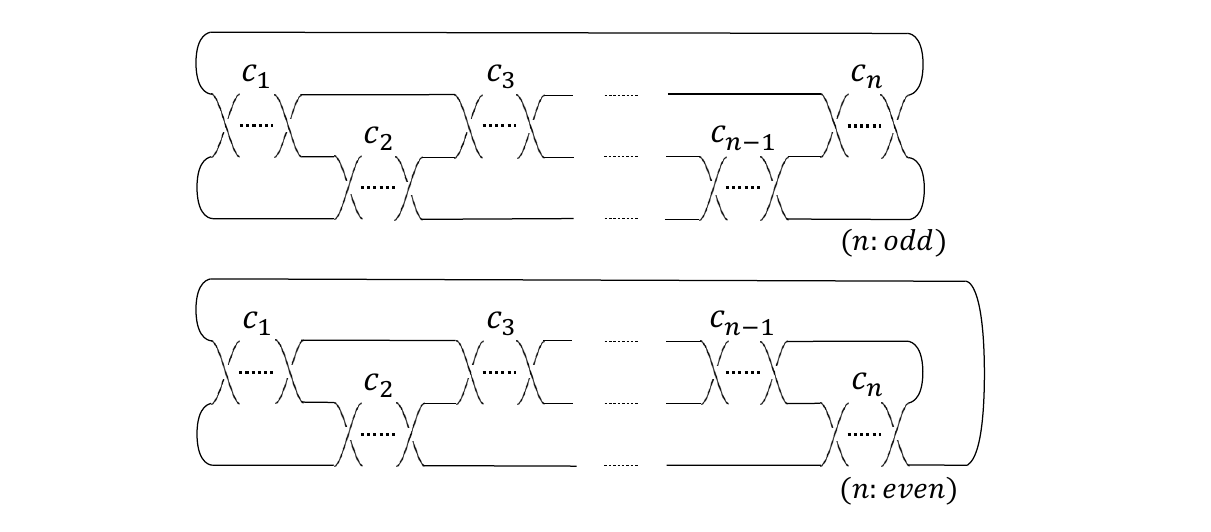}
 \caption{A two-bridge knot $C(c_1, c_2, ... , c_n)$.}
 \label{fig:twobridge}
\end{figure}
Previous work \cite{Naka3} computed the $\Delta$-unknotting numbers of certain families of two-bridge knots, as summarized in Propositions~\ref{prop:even} and \ref{prop:odd}.

\begin{proposition}[\cite{Naka3}]\label{prop:even}
Let $K=C(c_1, c_2, ... , c_n)$ be a two-bridge knot, where $c_i$ is a positive even integer for $1 \leq i \leq n$ and $n$ is even.
Then, we have 
\begin{equation*}
u^{\Delta}(K) = 
\frac{1}{4} \sum_{j=1}^{n/2} (\sum_{i=1}^{j} c_{2i-1}) c_{2j} \quad
\big(= -a_2(K)\big).
\end{equation*}
\end{proposition}

\begin{proposition}[\cite{Naka3}]\label{prop:odd}
Let $K=C(c_1, c_2, ... , c_n)$ be a two-bridge knot, where $c_i$ is a positive even integer for $1 \leq i \leq n-1$, and $c_n$ is a positive odd integer. In particular, if $n$ = 1 ($K=C(c_1)$), then $c_1$ is a positive odd integer.
Then, we have $u^{\Delta}(K)=|a_2(K)|$.
\begin{enumerate}
    \item
    If $n$ is even, then
    \[
    u^{\Delta}(K)=
\frac{1}{4} \sum_{j=1}^{n/2} (\sum_{i=1}^{j} c_{2i-1}) c_{2j} + \frac{1}{8} ( \sum_{i=1}^{n/2} c_{2i-1} )^{2} \quad \big(= a_2(K)\big).
\]
    \item
    If $n$ is odd, then
    \[
    u^{\Delta}(K)=
\frac{1}{4} \sum_{j=1}^{(n-1)/2} (\sum_{i=1}^{j} c_{2i-1}) c_{2j} + \frac{1}{8} \{ ( \sum_{i=1}^{(n+1)/2} c_{2i-1} )^{2} -1 \} \quad \big(= a_2(K)\big).
\]
\end{enumerate}
\end{proposition}

We shall use Proposition~\ref{prop:oddnew} throughout this paper.
It is obtained from Proposition~\ref{prop:odd}
by replacing the assumption that $c_n$ is odd with the condition that $c_1$ be odd.
Since the proof follows the same line of argument, we omit the details.

\begin{proposition}[\cite{Naka3}]\label{prop:oddnew}
Let $K=C(c_1, c_2, ... , c_n)$ be a two-bridge knot, where $c_1$ is a positive odd integer, and $c_i$ is a positive even integer for $2 \leq i \leq n$.
Then, we have $u^{\Delta}(K)=|a_2(K)|$.
\begin{enumerate}
    \item
    If $n$ is even, then
    \[
    u^{\Delta}(K)=
\frac{1}{4} \sum_{j=1}^{n/2} (\sum_{i=1}^{j} c_{2i-1}) c_{2j} + \frac{1}{8} ( \sum_{i=1}^{n/2} c_{2i} )^{2} \quad \big(= a_2(K)\big).
\]
    \item
    If $n$ is odd, then
    \[
    u^{\Delta}(K) =
\frac{1}{4} \sum_{j=1}^{(n-1)/2} (\sum_{i=1}^{j} c_{2i}) c_{2j+1} + \frac{1}{8} \{ ( \sum_{i=1}^{(n+1)/2} c_{2i-1} )^{2} -1 \} \quad \big(= a_2(K)\big).
\]
\end{enumerate}
\end{proposition}

\subsection{Pretzel Knots and Their $\Delta$-Unknotting Numbers} 
\label{subsec22}

\leavevmode

A pretzel link $P(p_1, p_2, ... , p_n)$ is a link obtained by connecting $n$ twisted bands, where each band has $p_i$ half-twists.
The sign of $p_i$ determines the direction of the $i$th twist region
(see Figure \ref{fig:pretzel}).
A pretzel link is a knot if and only if precisely one $p_i$ is even, or all $p_i$ are odd and $n$ is odd.

A positive pretzel knot is a pretzel knot whose standard diagram has either all positive crossings or all negative crossings, as follows. There are two types of positive pretzel knot: the odd type and the even type.
For a positive pretzel knot of odd type, both $n$ and all $p_i$ are odd and positive.
For a positive pretzel knot of even type, one of the following holds:
\begin{itemize}
    \item even type $A$ : $n$ is even and positive, $p_1$ is even and positive, and $p_i$ $(2 \leq i \leq n)$ are odd and positive; or
    \item even type $B$ : $n$ is odd and positive, $p_1$ is even and negative, and $p_i$ $(2 \leq i \leq n)$ are odd and positive.
\end{itemize}
\begin{figure}
    \centering
    \includegraphics[width=0.8\linewidth]{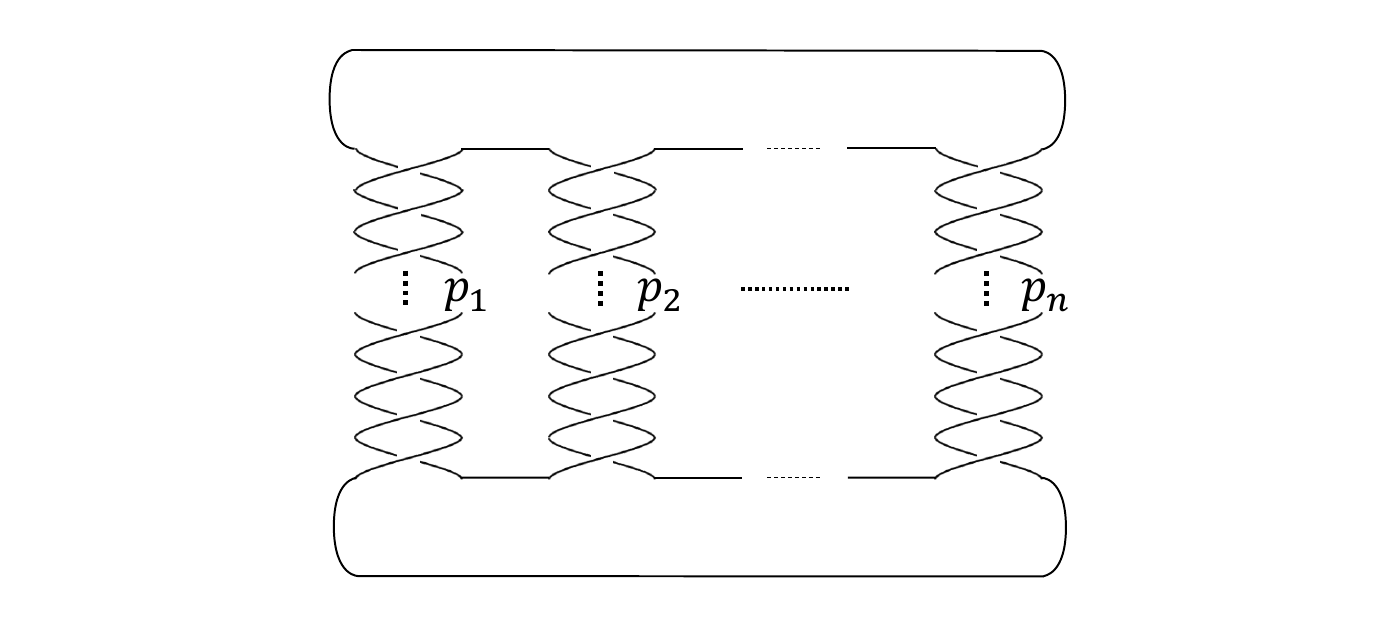}
    \caption{A pretzel knot $P(p_1, p_2, ... , p_n)$.}
    \label{fig:pretzel}
\end{figure}
Previous work \cite{Naka4} computed the $\Delta$-unknotting numbers of positive pretzel knots, as summarized in Propositions~\ref{prop:preodd} and \ref{prop:preeven}.

\begin{proposition}[\cite{Naka4}]\label{prop:preodd}
Let $P=P(p_1, p_2, ... , p_n)$ be a positive pretzel knot of odd type, where $p_i$ is a positive odd integer for $1 \leq i \leq n$ and $n$ is odd.
Then, we have
\begin{equation*}
u^{\Delta}(P) = 
\frac{1}{4} \sum_{1 \leq i<j }^{n} p_{i}p_{j} +  \frac{1}{8} (n-1) \quad \big(= a_2(P)\big).
\end{equation*}
\end{proposition}

\begin{proposition}[\cite{Naka4}]\label{prop:preeven}
Let $P=P(p_1, p_2, ... , p_n)$ be a positive pretzel knot of even type, where $p_1$ is an even integer, and $p_i$ is a positive odd integer for $2 \leq i \leq n$. 
Then, we have
\begin{enumerate}
    \item
    If $n$ is even and $p_1 > 0$, then
    \[
    u^{\Delta}(P)=
\frac{1}{8} \sum_{i=1}^{n} p_{i}^2 + \frac{1}{4} p_1 \sum_{i=2}^{n} p_i - \frac{1}{8} (n-1) \quad \big(= a_2(P)\big).
\]
    \item
    If $n$ is odd and $p_1 < 0$, then
    \[
    u^{\Delta}(P)=
\frac{1}{8} \sum_{i=2}^{n} p_{i}^2 - \frac{1}{4} p_1 \sum_{i=2}^{n} p_i - \frac{1}{8} (n-1) \quad \big(= a_2(P)\big).
\]
\end{enumerate}
\end{proposition}

\subsection{Montesinos Knots and Rational Tangles} 
\label{subsec23}

\leavevmode

A Montesinos knot $K = M\big(b ; (\alpha_1, \beta_1), (\alpha_2, \beta_2), \ldots, (\alpha_r, \beta_r)\big)$ with $r$ branches is a knot as illustrated in Figure \ref{fig:montesinos}.
Here $r$, $b$, $\alpha_k$, and $\beta_k$ are integers 
such that $r \geq 0$,  $\alpha_k \geq 2$, and $\gcd(\alpha_k, \beta_k)=1$.
Let $T(\beta_k / \alpha_k)$ denote the rational tangle of slope $\beta_k / \alpha_k$ (see Figure  \ref{fig:tangle}).
\begin{figure}[htbp]
    \centering
    \includegraphics[width=0.8\linewidth]{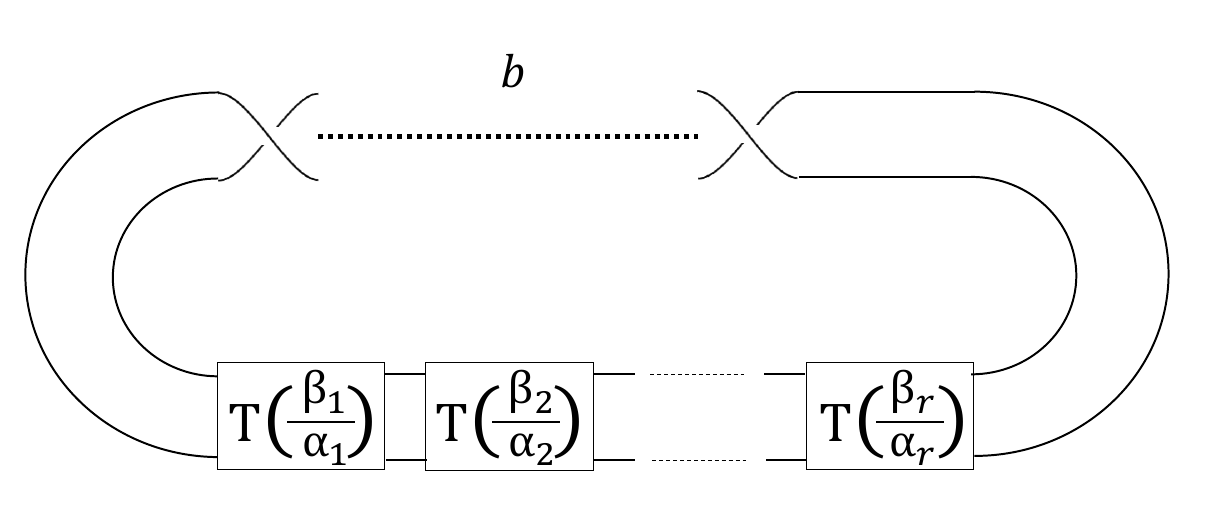}
    \caption{A Montesinos knot $M\big(b ; (\alpha_1, \beta_1), (\alpha_2, \beta_2), \ldots, (\alpha_r, \beta_r)\big)$.}
    \label{fig:montesinos}
\end{figure}
\begin{figure}[htbp]
    \centering
    \includegraphics[width=0.8\linewidth]{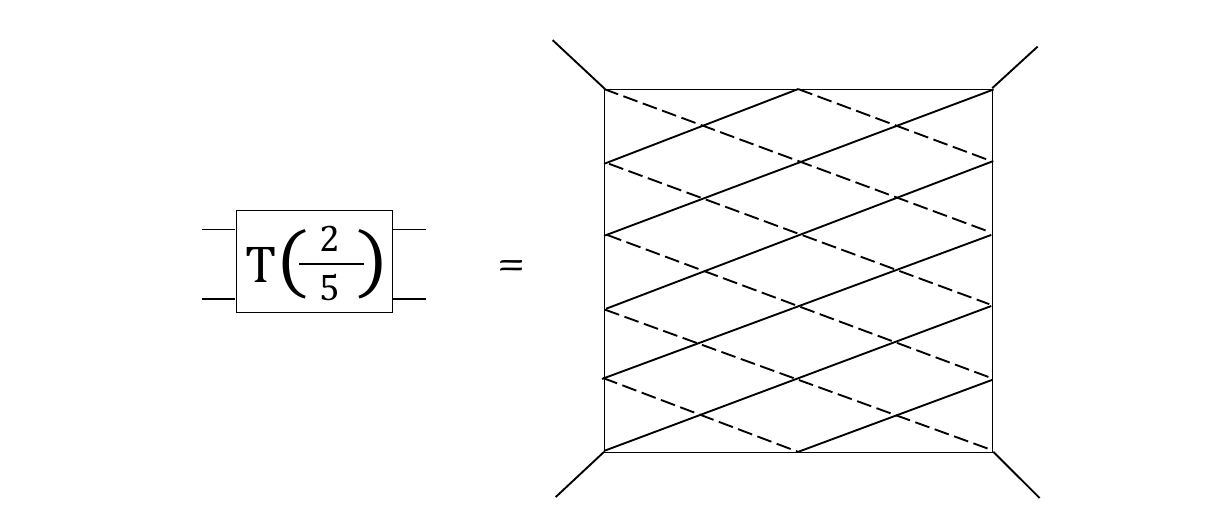}
    \caption{The tangle $T(2/5)$.}
    \label{fig:tangle}
\end{figure}
For the tangle $T(\beta_k / \alpha_k)$,
let $L \big( T(\beta_k / \alpha_k) \big)$ denote the link obtained by the standard left/right closure
(see Figure \ref{fig:closure}).
\begin{figure}[htbp]
    \centering
    \includegraphics[width=0.3\linewidth]{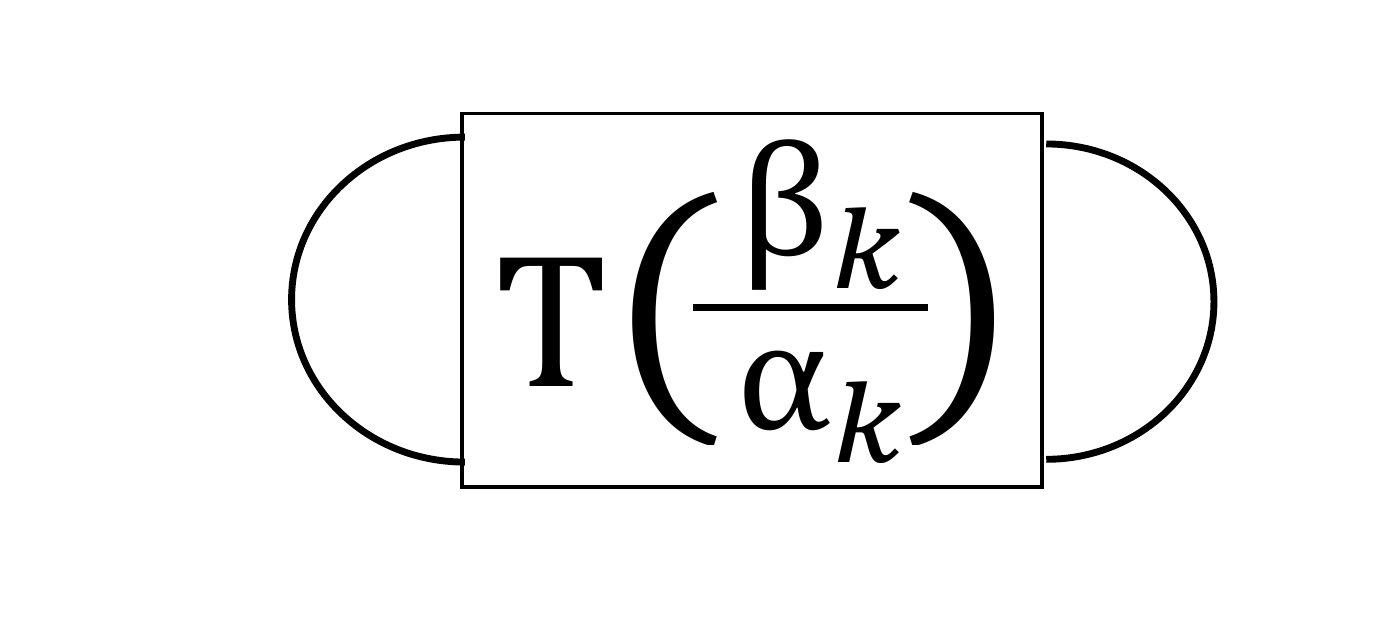}
    \caption{A link
    $L \big( T(\beta_k / \alpha_k) \big)$
    obtained by the standard left/right closure
    of the tangle $T(\beta_k / \alpha_k)$.}
    \label{fig:closure}
\end{figure}
A continued fraction expansion of a rational number $\beta / \alpha$, where we assume $-\alpha < \beta < \alpha$, is a finite sequence $s_1, s_2$, ... , $s_n$ such that
\begin{align*}
\frac{\beta}{\alpha}
& = \cfrac{1}{\,s_1 - \cfrac{1}{\,s_2 - \cfrac{1}{\,\ddots - \cfrac{1}{s_n}}}} \\
& = \cfrac{1}{\,s_1 + \cfrac{1}{\,- s_2 + \cfrac{1}{\,\ddots + \cfrac{1}{(-1)^{n-1} s_n}}}},
\end{align*}
and $s_i \neq 0$ for $1 \leq i \leq n.$
The tangle $T(\beta / \alpha)$ is then representable as in Figure \ref{fig:tangletwo} (see \cite{HaM}).
\begin{figure}[htbp]
    \centering
    \includegraphics[width=1\linewidth]{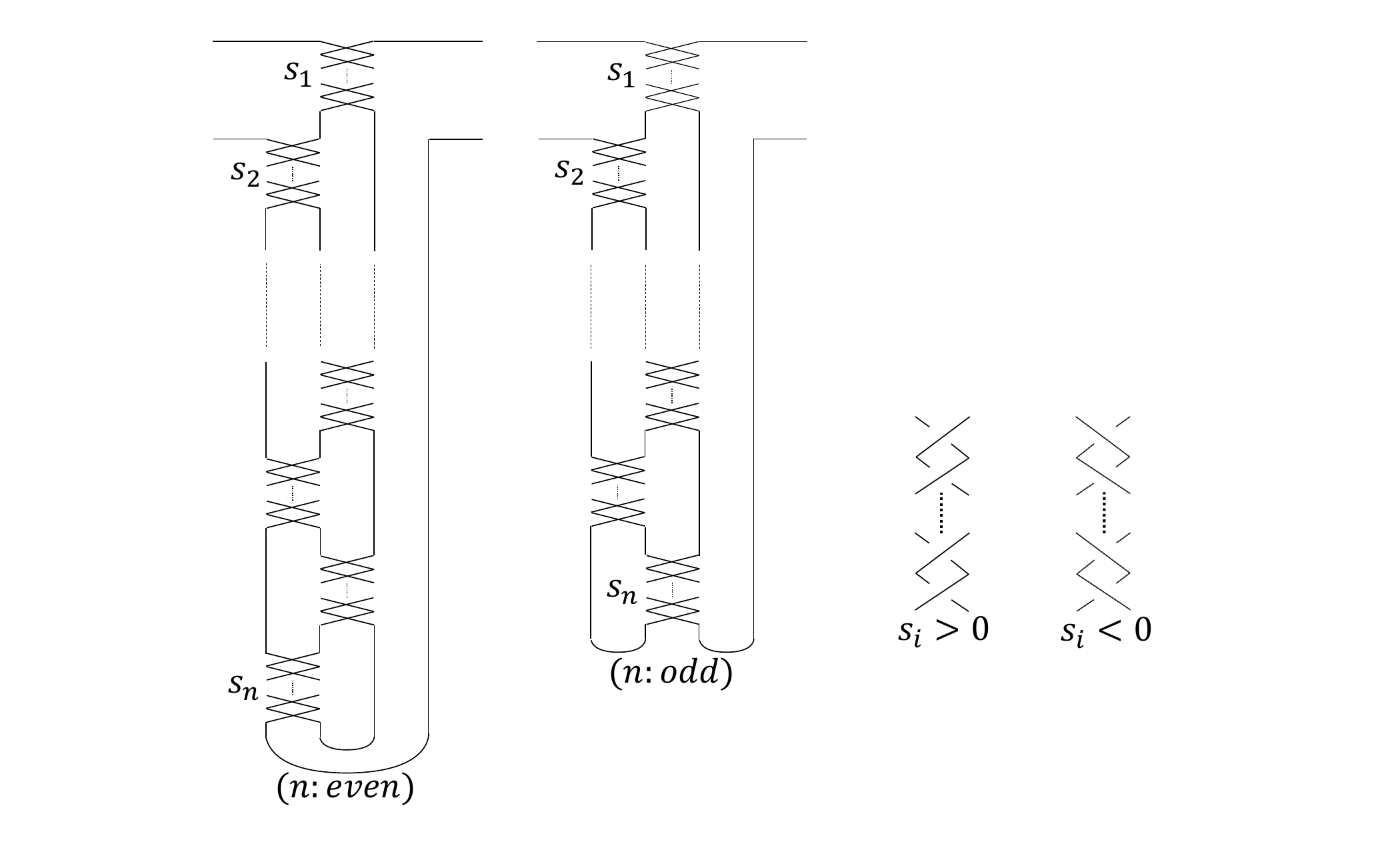}
    \caption{Another representation of a rational tangle for illustration.}
    \label{fig:tangletwo}
\end{figure}
In this paper, we consider the following Montesinos knot
$K = M(0; (\alpha_1, \beta_1), \dots, (\alpha_r, \beta_r))$,
where for each $1 \le k \le r$, the rational number $\beta_k/\alpha_k$
admits a finite continued fraction expansion
\[
\frac{\beta_k}{\alpha_k}
=
\cfrac{1}{c^{(k)}_1
 + \cfrac{1}{c^{(k)}_2
 + \cfrac{1}{\ddots
 + \cfrac{1}{c^{(k)}_{n(k)}}}}},
\]
where $n(k)$ is a positive integer, and $c^{(k)}_i \neq 0$ for $1 \leq i \leq n(k)$.

The tangle $T(\beta_k / \alpha_k)$ is represented as in Figure \ref{fig:tangletwoc}, 
a presentation adopted to facilitate comparison with two-bridge knots.
\begin{figure}[htbp]
    \centering
    \includegraphics[width=1\linewidth]{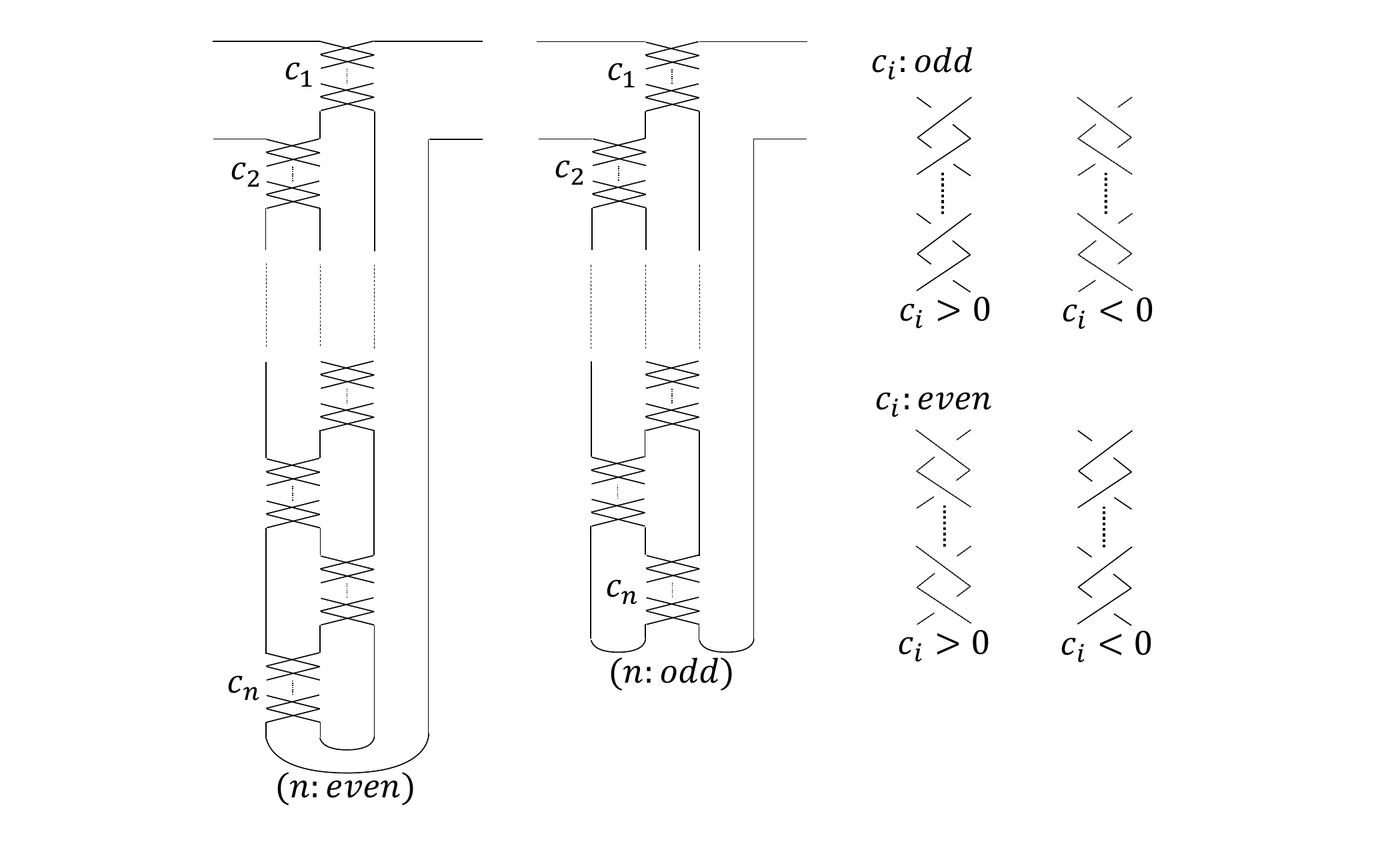}
    \caption{Another representation of a rational tangle used in this paper.}
    \label{fig:tangletwoc}
\end{figure}
\begin{claim}\label{claim:ABCDEF}
We consider six types of tangles as shown in Figure \ref{fig:tangleetc}: Type A, Type B, Type C, Type D, Type E, and Type F.

A rational number $\beta/\alpha$ admits a finite continued fraction expansion
\[
\frac{\beta}{\alpha} = \cfrac{1}{c_1 + \cfrac{1}{c_2 + \dots + \cfrac{1}{c_n}}},
\]
where $n$ is a positive integer, and $c_i \neq 0$ for $1 \leq i \leq n$.

Here we define
\begin{align*}
& V := \frac{1}{4} \sum_{j=1}^{n/2}
  ( \sum_{i=1}^{j} c_{2i-1} ) c_{2j} 
  \quad \text{for even $n$}, \\
& W := \frac{1}{4} \sum_{j=1}^{(n-1)/2}
  ( \sum_{i=1}^{j} c_{2i} ) c_{2j+1}
  \quad \text{for odd $n$}, \\
& X := \frac{1}{2} \sum_{i=1}^{(n+1)/2} c_{2i-1}
  \quad \text{for odd $n$}, \\
& Y := \frac{1}{2} \sum_{i=1}^{n/2} c_{2i}
  \quad \text{for even $n$}.
\end{align*}


We define $\beta'/\alpha'$ depending on the tangle types by
\[
\frac{\beta'}{\alpha'} : =
\begin{cases}
\displaystyle \frac{1}{ 0 + \frac{1}{2Y}}, & \text{if Type $A$},\\[1mm]
\displaystyle \frac{\beta}{\alpha}, & \text{if Type $B$},\\[1mm]
\displaystyle \frac{1}{2X}, & \text{if Type $C,D,E,$ and $F$}.\\
\end{cases}
\]

By the same argument as in the proof of Claim~\ref{claim:technique2} (see below), $T(\beta/\alpha)$ can be deformed into $T(\beta'/\alpha')$ by $\Delta$-moves.

\medskip
(1) \textbf{Type A.}  

A tangle $T(\beta/\alpha)$ is of Type A if all $c_i$ are positive even, and $n$ is even. 

$T(\beta/\alpha)$ can be deformed into $T(\beta'/\alpha')$ by $V$ times $\Delta$-moves.

By Proposition \ref{prop:even}, $u^{\Delta}(L \big( T(\beta / \alpha) \big)) = V$.

\medskip
(2) \textbf{Type B.}  

A tangle $T(\beta/\alpha)$ is of Type B if $c_1$ is positive odd, $c_i$ ($i \ge 2$) are positive even, and $n$ is even. 


By Proposition \ref{prop:oddnew}(1), $u^{\Delta}(L \big( T(\beta / \alpha) \big)) = V + \frac{1}{2}Y^2$.

\medskip
(3) \textbf{Type C.}  

A tangle $T(\beta/\alpha)$ is of Type C if $c_1$ is positive odd, $c_i$ ($i \ge 2$) are positive even, and $n$ is odd. 

$T(\beta/\alpha)$ can be deformed into $T(\beta'/\alpha')$ by $W$ times $\Delta$-moves.

By Proposition \ref{prop:oddnew}(2), $u^{\Delta}(L \big( T(\beta / \alpha) \big)) = W + \frac{1}{2}X^2 - \frac{1}{8}$.

\medskip
(4) \textbf{Type D.}  

A tangle $T(\beta/\alpha)$ is of Type D if $c_1$ is positive odd, $c_{2i-1}$ ($i \ge 2$) are positive even, $c_{2i}$ ($i \ge 1$) are negative even, and $n$ is odd. 

$T(\beta/\alpha)$ can be deformed into $T(\beta'/\alpha')$ by $-W$ times $\Delta$-moves.

\medskip
(5) \textbf{Type E.} 

A tangle $T(\beta/\alpha)$ is of Type E if all $c_i$ are positive even, and $n$ is odd. 

$T(\beta/\alpha)$ can be deformed into $T(\beta'/\alpha')$ by $W$ times $\Delta$-moves.

\medskip
(6) \textbf{Type F.} 

A tangle $T(\beta/\alpha)$ is of Type F if $c_{2i-1}$ are negative even, $c_{2i}$ are positive even, and $n$ is odd. 

$T(\beta/\alpha)$ can be deformed into $T(\beta'/\alpha')$ by $-W$ times $\Delta$-moves.
\end{claim}
\begin{figure}[htbp]
    \centering
    \includegraphics[width=1\linewidth]{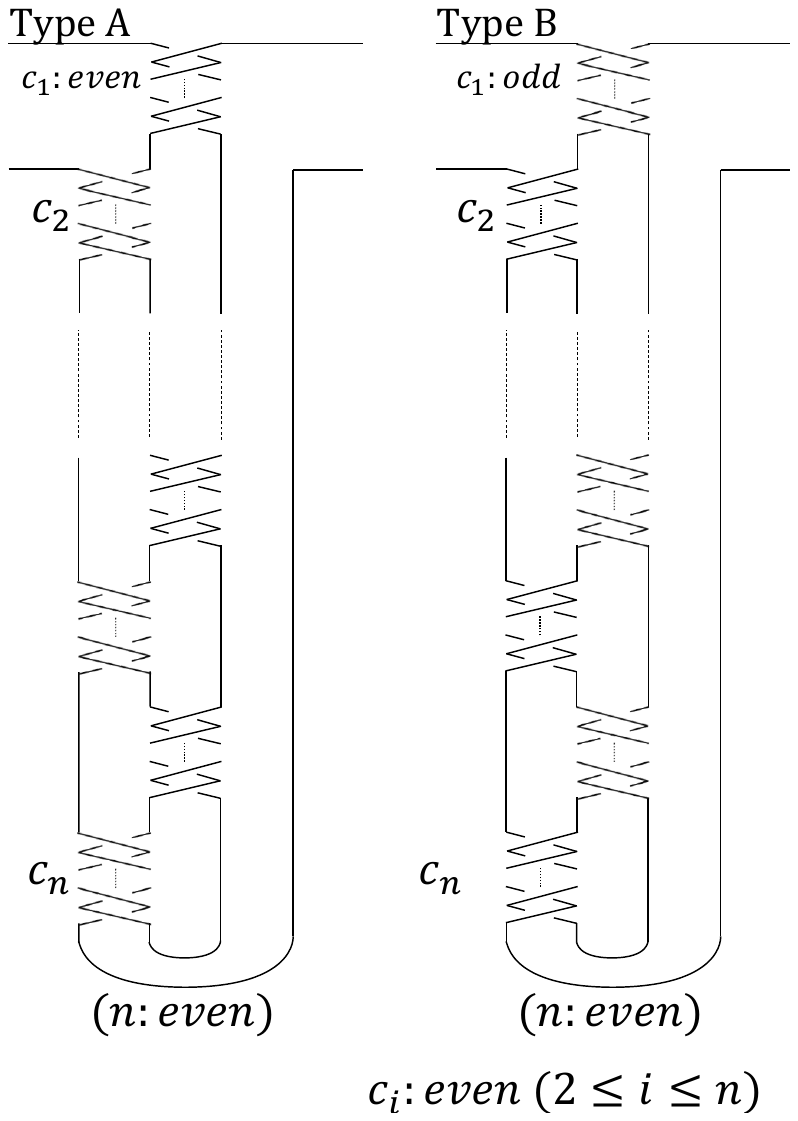}
    \caption{Tangles of Type A and B.}
    \label{fig:tangleetc}
\end{figure}
\begin{figure}[htbp]
    \centering
    \includegraphics[width=1\linewidth]{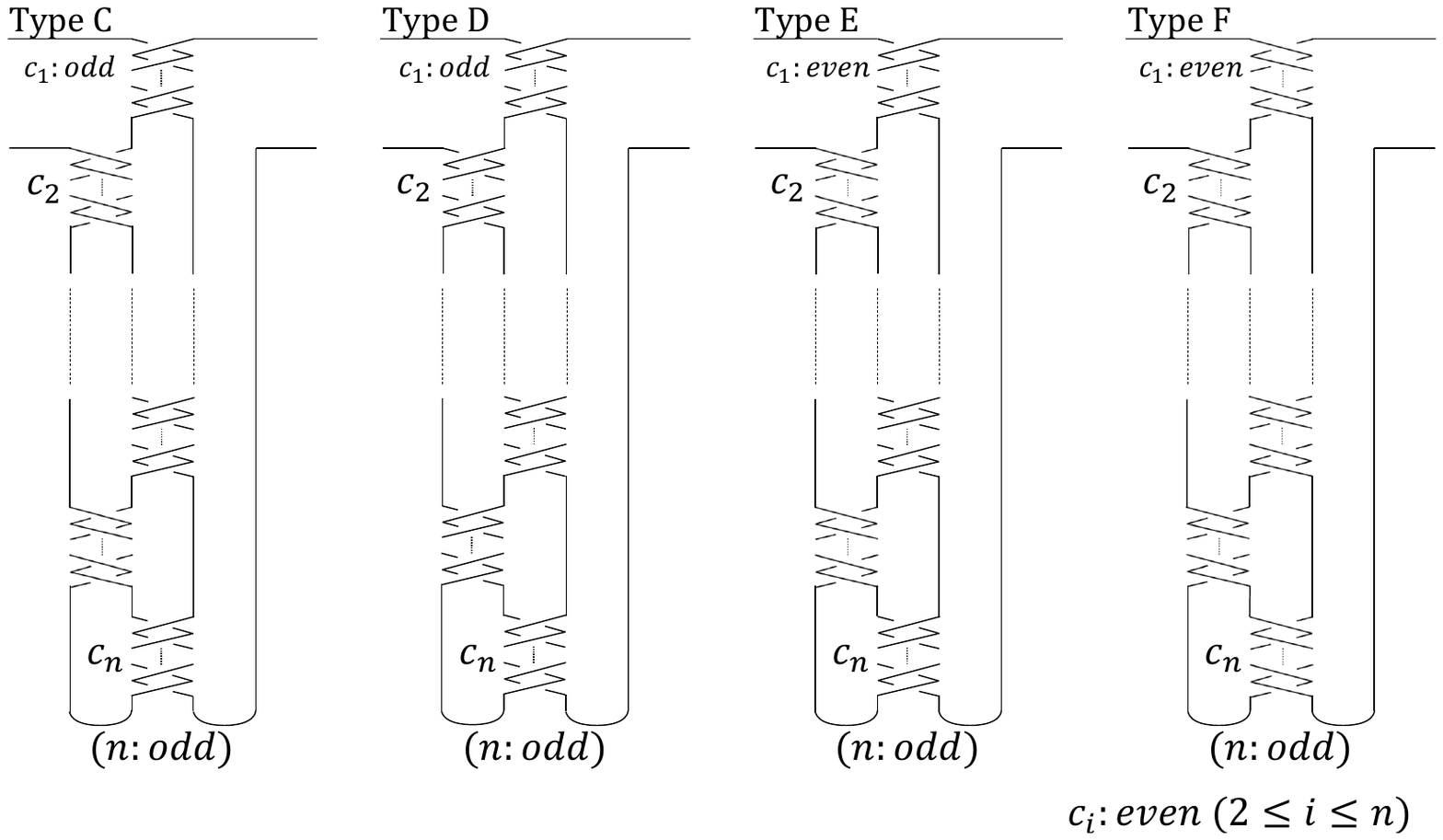}
    \caption{Tangles of Type C--F.}
    \label{fig:tangleetc}
\end{figure}
\begin{figure}[htbp]
    \centering
    \includegraphics[width=1\linewidth]{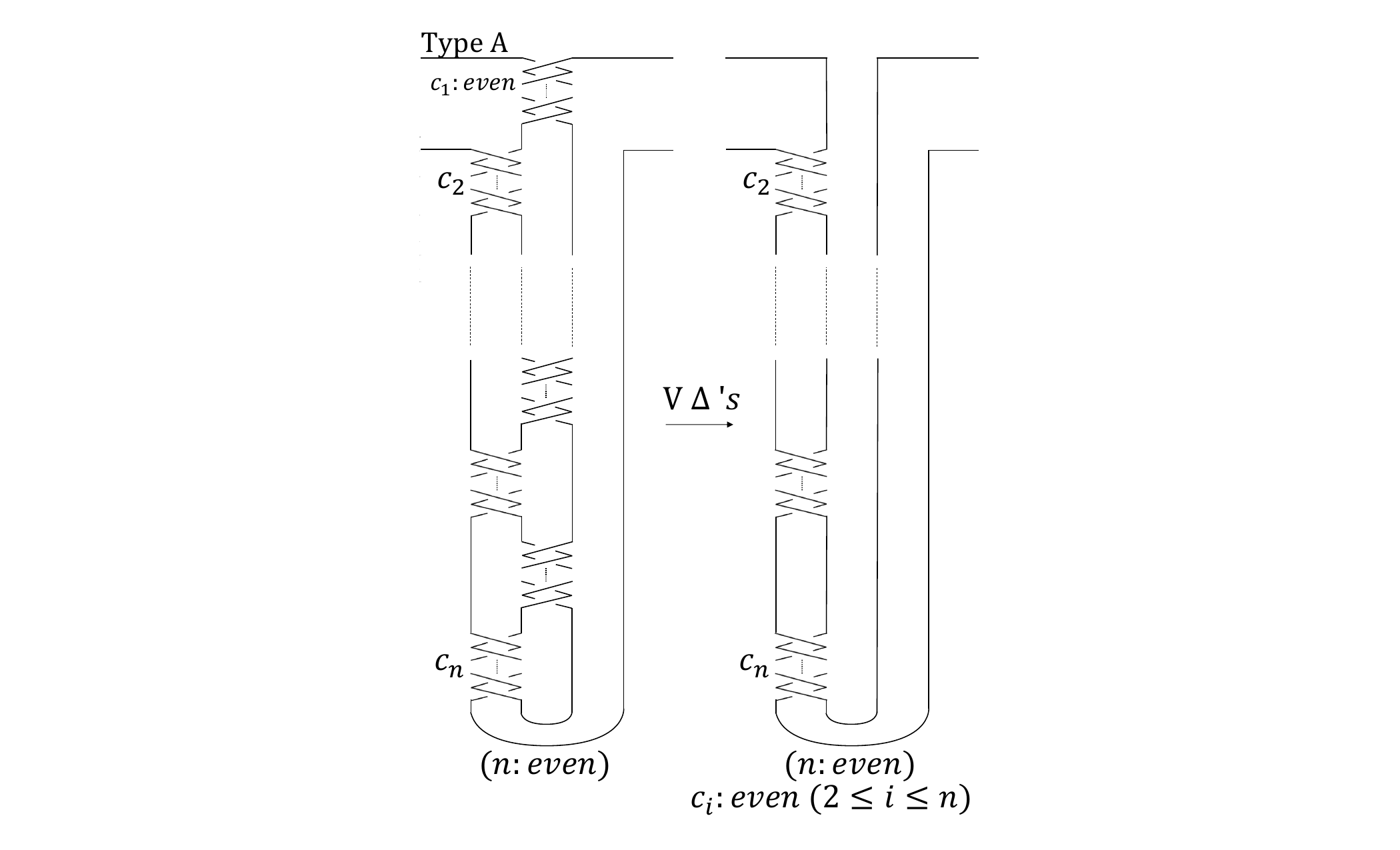}
    \caption{A deformation of a tangle of Type A.}
    \label{fig:tangleAto}
\end{figure}
\begin{figure}[htbp]
    \centering
    \includegraphics[width=1\linewidth]{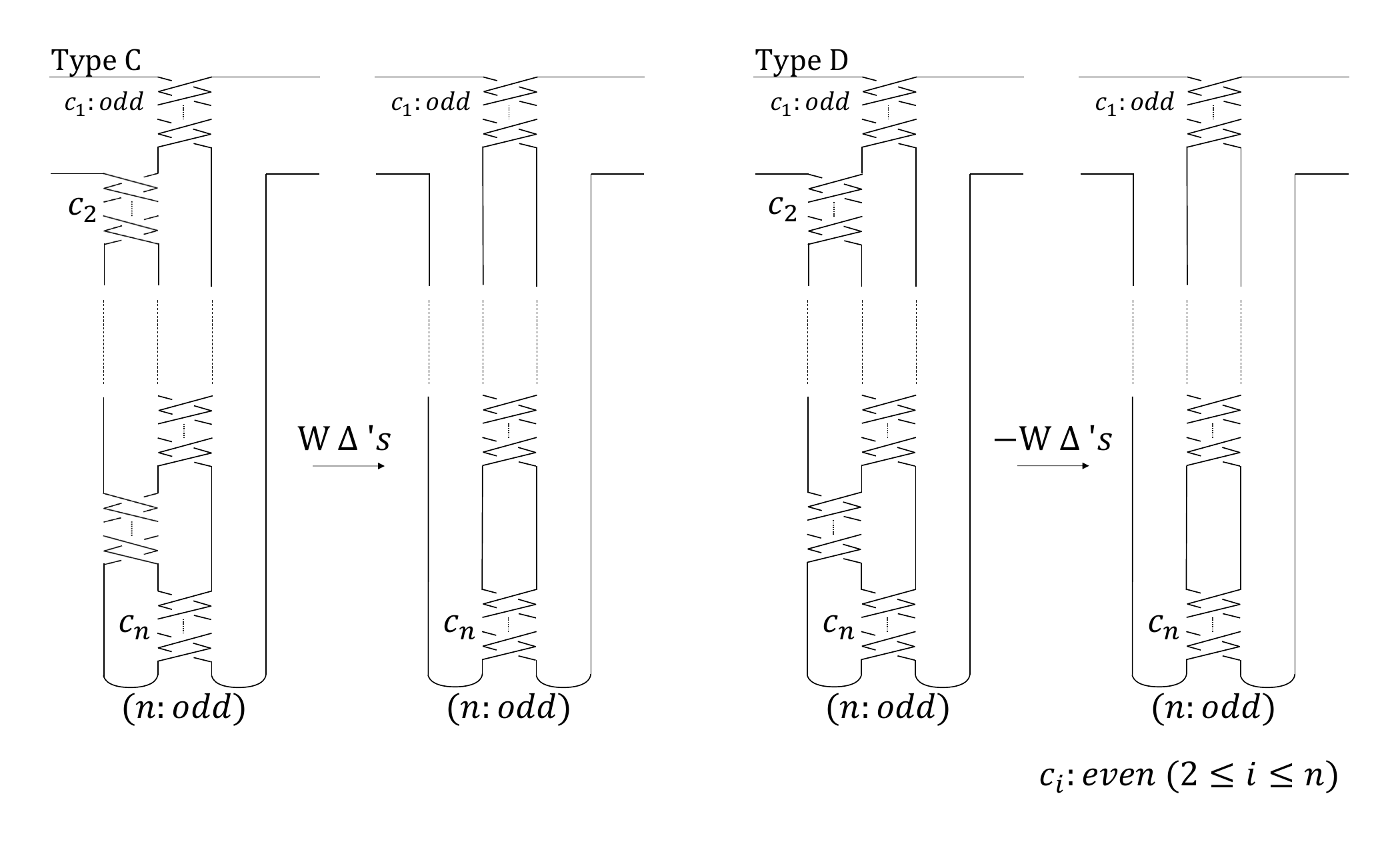}
    \caption{Deformations of tangles of Type C and D.}
    \label{fig:tangleCto}
\end{figure}
\begin{figure}[htbp]
    \centering
    \includegraphics[width=1\linewidth]{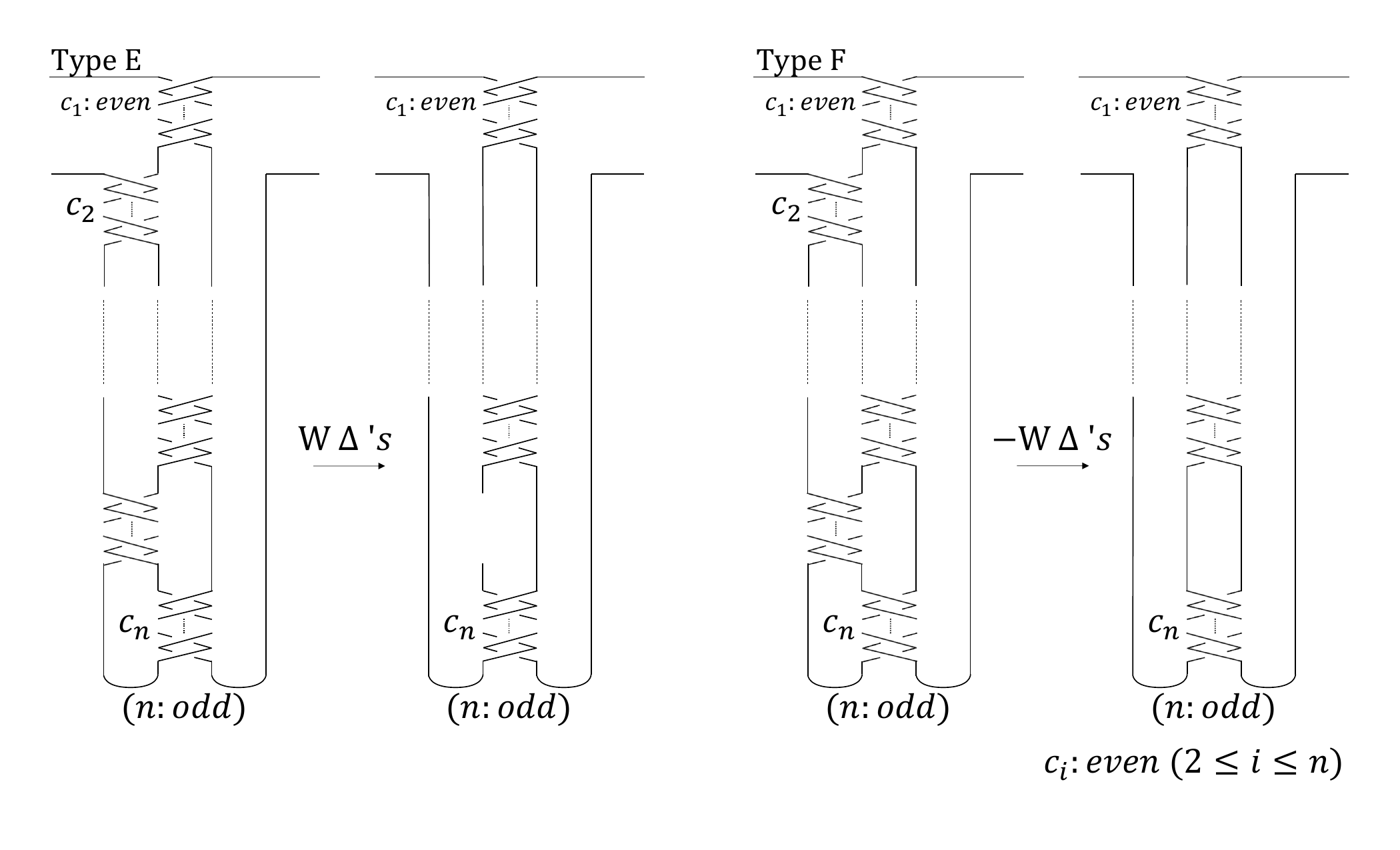}
    \caption{Deformations of tangles of Type E and F.}
    \label{fig:tangleEto}
\end{figure}
Table~\ref{tab:tangle-types} lists the characteristics of Type A--F tangles. 
The main differences lie in the parity and sign of $c_1$ and the signs of the remaining $c_i$. 

\begin{table}[h]
\centering
\caption{Summary of Type A--F tangles with exact conditions}
\label{tab:tangle-types}
\begin{tabular}{c c c c c c c}
\hline
Type & $c_1$ & $c_{2i-1}$ ($i\ge 2$) & $c_{2i}$ ($i\ge 1$) & $n$ & $\beta'/\alpha'$ & $\Delta$-move \\
\hline
A & positive even & positive even & positive even & even & $2Y$ & $+V$ \\
B & negative odd & negative even & negative even & even & $-$ & $-$ \\
C & positive odd & positive even & positive even & odd & $1/2X$ & $+W$ \\
D & positive odd & positive even & negative even & odd & $1/2X$ & $-W$ \\
E & positive even & positive even & positive even & odd & $1/2X$ & $+W$ \\
F & negative even & negative even & positive even & odd & $1/2X$ & $-W$ \\
\hline
\end{tabular}
\end{table}

\vspace{1em}
The upper left and lower left corners of $T(\beta / \alpha)$ are connected by a strand if and only if $\alpha$ is even.
Hence we see that if $K = M\big(0 ; (\alpha_1, \beta_1), (\alpha_2, \beta_2), \ldots, (\alpha_r, \beta_r)\big)$
is to be a knot rather than a link, at most one of $\alpha_1, \alpha_2$, ... , $\alpha_r$ can be even.
Since a cyclic permutation indices does not change the knot type, we hereafter assume that $\alpha_2$, ..., $\alpha_r$ are odd.
With this convention, we say that $K$ is of odd type if $\alpha_1$ is odd, and of even type if $\alpha_1$ is even.

Let $T(\beta_1 / \alpha_1), \dots, T(\beta_r / \alpha_r)$ be tangles, and let $L$ be the associated Montesinos link. 
As one example, if $L$ is a knot, it is of odd type when all tangles are of Type A--D, 
and it is of even type when $T(\beta_1 / \alpha_1)$ is of Type E or F and the remaining tangles are of Type A--D (see Figure \ref{fig:montetype}). 
Other arrangements of tangles can also yield Montesinos knots of odd or even type. 
Whether $L$ is a knot depends on the number of tangles of each type.
\begin{figure}[htbp]
    \centering
    \includegraphics[width=1\linewidth]{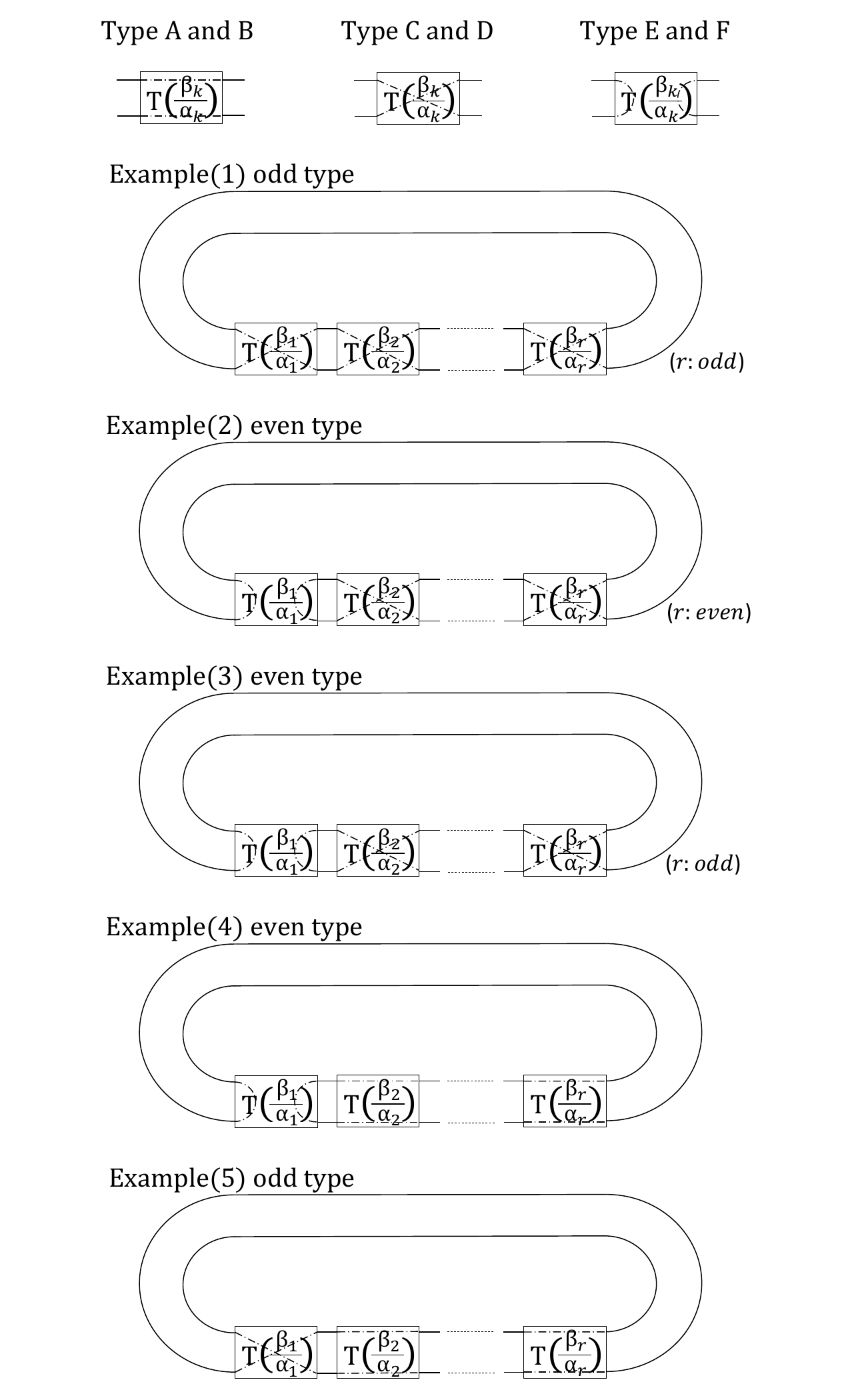}
    \caption{Montesinos knots of odd and even types.}
    \label{fig:montetype}
\end{figure}
\subsection{Techniques and Propositions} 
\label{subsec24}

\leavevmode

We use the technique in \cite{Naka3}.

\begin{claim}[\cite{Naka3}]\label{claim:technique}
Let $K=C(c_1, c_2, ... , c_n)$ be a two-bridge knot, where $c_1$ is a positive even integer and $c_2$ is a positive integer.
In Figure \ref{fig:cct}, if we perform $\frac{1}{2} c_1$ times $\Delta$-moves, we can exchange the crossing labeled with the asterisk *.
Then, we have 
$d_G^{\Delta}\big(C(c_1, \underline{c_2}, c_3, ... , c_n), C(c_1, \underline{c_2-2}, c_3, ... , c_n)\big) \leq \frac{1}{2} c_1$.
\end{claim}
\begin{figure}[htbp]
 \centering
 \includegraphics[width=0.8\linewidth]{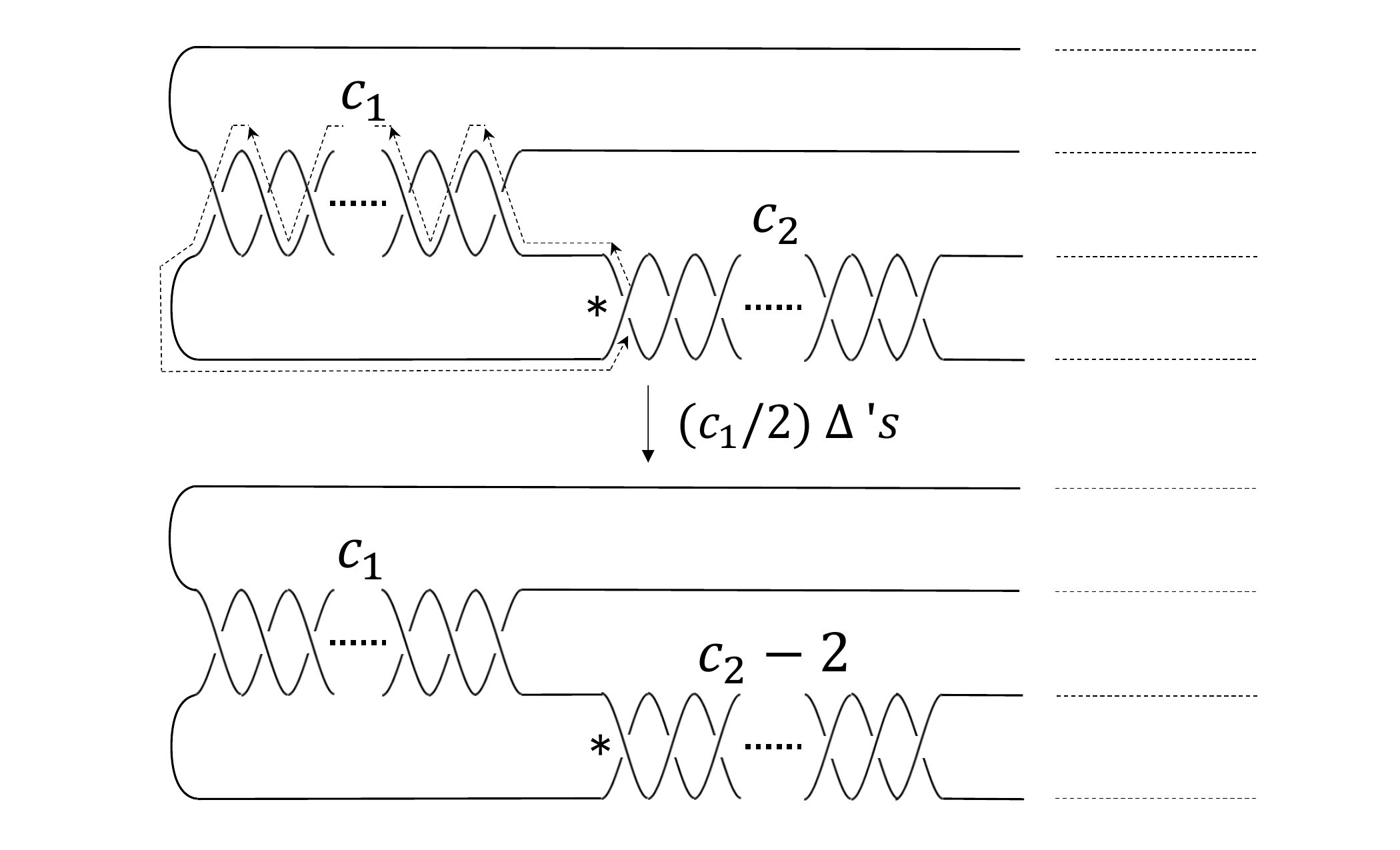}
 \caption{A crossing change realized by $\frac{1}{2} c_1$ $\Delta$-moves.}
 \label{fig:cct}
\end{figure}
By performing the same computation, we obtain Claim \ref{claim:technique2}.

\begin{claim}\label{claim:technique2}
Let $K=C(c_1, c_2, ... , c_n)$ be a two-bridge knot, where $c_1$ is a positive integer and $c_i$ is a positive even integer for $2 \leq i \leq n$. Then, we have

\begin{enumerate}
    \item If $n$ is even and $c_1$ is even, then
    \[
    d_G^{\Delta}\big(C(c_1, c_2, c_3, ... , c_n), C(0, c_2 + c_4 + ... + c_n)\big) 
    \leq \frac{1}{4} \sum_{j=1}^{n/2}
    (\sum_{i=1}^{j} c_{2i-1} ) c_{2j}
    = V.
    \]
    \item If $n$ is odd and $c_1$ is odd, then
    \[
    d_G^{\Delta}\big(C(c_1, c_2, c_3, ... , c_n), C(c_1 + c_3 + ... + c_n)\big)
    \leq \frac{1}{4} \sum_{j=1}^{(n-1)/2}
    ( \sum_{i=1}^{j} c_{2i} ) c_{2j+1}
    = W.
    \]
\end{enumerate}
\end{claim}

\begin{proof}
(1)
Performing the same computation in Claim~\ref{claim:technique}, we have
\begin{align*}
d_G^{\Delta}\big(C(c_1, ... , c_{n-1}, c_n), 
C(c_1, ... , c_{n-1} - 2, c_n)\big) \leq \frac{1}{2}c_n. 
\end{align*}
By repeating the same computation $\frac{1}{2} c_{n-1}$ times, we have
\begin{gather*}
d_G^{\Delta}\big(C(c_1, ... , c_{n-2}, \underline{c_{n-1}}, c_n), C(c_1, ... , c_{n-2}, \underline{c_{n-1} - 2}, c_n)\big) 
\leq \frac{1}{2} c_n \quad, ... , \\
d_G^{\Delta}\big(C(c_1, ... , c_{n-2}, \underline{2}, c_n), C(c_1, ... , c_{n-2}, \underline{0}, c_n)\big) \leq \frac{1}{2} c_n . 
\end{gather*}
By continuing with another $\frac{1}{2} c_{n-3}$ steps, we have
\begin{gather*}
d_G^{\Delta}\big(C(c_1, ... , c_{n-4}, \underline{c_{n-3}}, c_{n-2} + c_n), C(c_1, ... , c_{n-4}, \underline{c_{n-3} - 2}, c_{n-2} + c_n)\big) \\
\leq \frac{1}{2} (c_{n-2} + c_n) \quad, ... , \\
d_G^{\Delta}\big(C(c_1, ... , c_{n-4}, \underline{2}, c_{n-2} + c_n), C(c_1, ... , c_{n-4}, \underline{0}, c_{n-2} + c_n)\big) \\
\leq \frac{1}{2} (c_{n-2} + c_n). 
\end{gather*}
Proceeding further, we have 
\begin{gather*}
d_G^{\Delta}\big(C(\underline{c_1}, c_2 + ... + c_{n-2} + c_n), C(\underline{c_{1} - 2}, c_2 + ... + c_{n-2} + c_n) \big) \\
\leq \frac{1}{2} (c_2 + ... + c_{n-2} + c_n) \quad, ... , \\
d_G^{\Delta}\big(C(\underline{2}, c_2 + ... + c_{n-2} + c_n), C(\underline{0}, c_2 + ... + c_{n-2} + c_n)\big)  \\
\leq \frac{1}{2} (c_2 + ... + c_{n-2} + c_n) .
\end{gather*}
By summing up, we obtain 
\begin{align*}
& d_G^{\Delta}\big(C(c_1, ..., c_{n-2}, c_{n-1}, c_n), C(0, c_2 + ... + c_{n-2} + c_n)\big) \\
& \leq d_G^{\Delta}\big(C(c_1, ..., c_{n-2}, c_{n-1}, c_n), C(c_1, ..., c_{n-2}, c_{n-1} - 2, c_n)\big) + ... \\
& + d_G^{\Delta}\big(C(2, c_2 + ... + c_{n-2} + c_n), C(0, c_2 + ... + c_{n-2} + c_n)\big) \\
& \leq \frac{1}{4} \{ c_1 (c_2 + c_4 +... + c_{n-2} + c_n) + ... + c_{n-3} (c_{n-2} + c_n) + c_{n-1} c_n \} \\
& = \frac{1}{4} \{ c_1 c_2 + (c_1 + c_3) c_4 + ... + (c_1 + c_3 +... + c_{n-3} + c_{n-1}) c_n \} \\
& = \frac{1}{4} \sum_{j=1}^{n/2}
(\sum_{i=1}^{j} c_{2i-1}) c_{2j} = V.
\end{align*}

(2)
Same way, by Claim \ref{claim:technique}, 
we have
\begin{align*}
d_G^{\Delta}\big(C(c_1, ... , c_{n-1}, c_n), 
C(c_1, ... , c_{n-1} - 2, c_n)\big) \leq \frac{1}{2}c_n. 
\end{align*}
By repeating the same computation $\frac{1}{2} c_{n-1}$ times, we have
\begin{gather*}
d_G^{\Delta}\big(C(c_1, ... , c_{n-2}, \underline{c_{n-1}}, c_n), C(c_1, ... , c_{n-2}, \underline{c_{n-1} - 2}, c_n)\big) 
\leq \frac{1}{2} c_n \quad, ... , \\
d_G^{\Delta}\big(C(c_1, ... , c_{n-2}, \underline{2}, c_n), C(c_1, ... , c_{n-2}, \underline{0}, c_n)\big) \leq \frac{1}{2} c_n . 
\end{gather*}
By continuing with another $\frac{1}{2} c_{n-3}$ steps, we have
\begin{gather*}
d_G^{\Delta}\big(C(c_1, ... , c_{n-4}, \underline{c_{n-3}}, c_{n-2} + c_n), C(c_1, ... , c_{n-4}, \underline{c_{n-3} - 2}, c_{n-2} + c_n)\big) \\
\leq \frac{1}{2} (c_{n-2} + c_n) \quad, ... , \\
d_G^{\Delta}\big(C(c_1, ... , c_{n-4}, \underline{2}, c_{n-2} + c_n), C(c_1, ... , c_{n-4}, \underline{0}, c_{n-2} + c_n)\big) \\
\leq \frac{1}{2} (c_{n-2} + c_n). 
\end{gather*}
Proceeding further, we have 
\begin{gather*}
d_G^{\Delta}\big(C(c_1, \underline{c_2}, c_3 + ... + c_{n-2} + c_n), C(c_1, \underline{c_{2} - 2}, c_2 + ... + c_{n-2} + c_n) \big) \\
\leq \frac{1}{2} (c_3 + ... + c_{n-2} + c_n) \quad, ... , \\
d_G^{\Delta}\big(C(c_1, \underline{2}, c_3 + ... + c_{n-2} + c_n), C(c_1, \underline{0}, c_3 + ... + c_{n-2} + c_n)\big)  \\
\leq \frac{1}{2} (c_3 + ... + c_{n-2} + c_n) .
\end{gather*}
By summing up, we obtain 
\begin{align*}
& d_G^{\Delta}\big(C(c_1, ..., c_{n-2}, c_{n-1}, c_n), C(c_1 + c_3 + ... + c_{n-2} + c_n)\big) \\
& \leq d_G^{\Delta}\big(C(c_1, ..., c_{n-2}, c_{n-1}, c_n), C(c_1, ..., c_{n-2}, c_{n-1} - 2, c_n)\big) + ... \\
& + d_G^{\Delta}\big(C(c_1, 2, c_3 + ... + c_{n-2} + c_n), C(c_1, 0, c_3 + ... + c_{n-2} + c_n)\big) \\
& \leq \frac{1}{4} \{ c_2 (c_3 + c_5 +... + c_{n-2} + c_n) + ... + c_{n-3} (c_{n-2} + c_n) + c_{n-1} c_n \} \\
& = \frac{1}{4} \{ c_2 c_3 + (c_2 + c_4) c_5 + ... + (c_2 + c_4 +... + c_{n-3} + c_{n-1}) c_n \} \\
& = \frac{1}{4}
\sum_{j=1}^{(n-1)/2}
(\sum_{i=1}^{j} c_{2i}) c_{2j+1} = W.
\end{align*}
\end{proof}

In this paper, we use Propositions \ref{prop:lk}, \ref{prop:a2}.

\begin{proposition}\label{prop:lk}
Let $k_+$ be a knot, $k_-$ the knot obtained from $k_+$ by exchanging a positive crossing into a negative crossing, and $k_0$ a $2$-component link obtained from $k_+$ by smoothing at the crossing. Then $a_2(k_+)-a_2(k_-) = lk(k_0)$.
\end{proposition}

The proof can be found in \cite{Kau}(Chap. III).

\begin{proposition}[\cite{Oka1}]\label{prop:a2}
For any two knots $K$ and $K^{'}$, the difference $d_G^{\Delta}(K,K^{'}) - |a_2(K) - a_2(K^{'})|$ is a non-negative even integer. 
In particular, the difference $u^{\Delta}(K) -|a_2(K)|$ is also a non-negative even integer.
\end{proposition}

Since $a_2(K)=a_2(K^*)$ and $u^{\Delta}(K)=u^{\Delta}(K^*)$, where $K^*$ denotes the mirror image of a knot $K$, we do not distinguish a knot from its mirror image in this paper.

\vspace{1em}
\section{Proof of Theorem \ref{thm:main}}\label{sec3}
This section is devoted to the proof of Theorem \ref{thm:main}. 

Here, we introduce the following quantities to simplify the presentation of the expressions:
\begin{align*}
&V^{(k)} : = \frac{1}{4} \sum_{j=1}^{n/2}
(\sum_{i=1}^{j} c^{(k)}_{2i-1} ) c^{(k)}_{2j}, \quad
W^{(k)} : = \frac{1}{4} \sum_{j=1}^{(n-1)/2}
(\sum_{i=1}^{j} c^{(k)}_{2i} ) c^{(k)}_{2j+1} \\  
& X^{(k)} : = \frac{1}{2} \sum_{i=1}^{(n+1)/2} c^{(k)}_{2i-1}, \quad
Y^{(k)} : =  \frac{1}{2} \sum_{i=1}^{n/2} c^{(k)}_{2i}.
\end{align*}

For $k=1$, we may replace this expansion by any equivalent finite
continued fraction expansion whenever it is convenient for later arguments.
In particular, for simplicity of notation, we write
\[
\frac{\beta_1}{\alpha_1}
=
\cfrac{1}{c_1
 + \cfrac{1}{c_2
 + \cfrac{1}{\ddots
 + \cfrac{1}{c_n}}}},
\]
where $n = n(1)$ and $c_i = c^{(1)}_i$ for $1 \le i \le n$.

For simplicity of notation in the proof, we denote
\[
K =
M\Big(0; (\alpha_1, \beta_1), \dots, (\alpha_r, \beta_r)\Big)
:= MC_1(c_1, c_2, \dots, c_n),
\]
where the right-hand side denotes the presentation determined by
the continued fraction expansion of $\beta_1/\alpha_1$.

We make use of the knot diagram depicted in Figure \ref{fig:montetype}.

\main*

\subsection{Proof of (1)} 
\label{subsec31}

\leavevmode

\begin{proof}
(1)
Let $k_+ = MC_1(c_1, c_2, ... , c_{n-2}, c_{n-1}+2, c_n)$ and $K = k_- = MC_1(c_1, c_2, ... , c_{n-2}, c_{n-1}, c_n)$. 
By Proposition \ref{prop:lk}, 
$a_2(MC_1(c_1, c_2, ... , c_{n-2}, c_{n-1}+2, c_n)) - a_2(MC_1(c_1, c_2, ... , c_{n-2}, c_{n-1}, c_n)) = lk(k_0) = - \frac{1}{2} c_n$.

By repeating the computation $- \frac{1}{2}c_{n-1}$ times, we have
\begin{align*}
a_2(MC_1(c_1, c_2, ... , c_{n-2}, c_{n-1}+2, c_n))
& - a_2(MC_1(c_1, c_2, ... , c_{n-2}, c_{n-1}, c_n)) 
=  - \frac{1}{2} c_n \quad, ... , \\
a_2(MC_1(c_1, c_2, ... , c_{n-2}, 0, c_n)) 
& - a_2(MC_1(c_1, c_2, ... , c_{n-2}, -2, c_n))
= - \frac{1}{2} c_n .
\end{align*}

Here, $MC_1(c_1, c_2, ... , c_{n-3}, c_{n-2}, 0, c_n) \cong MC_1(c_1, c_2, ... ,c_{n-3}, c_{n-2} + c_n)$.

By continuing with another $- \frac{1}{2} c_{n-3}$ steps, we have
\begin{align*}
a_2(MC_1(c_1, c_2, ... , c_{n-4}, c_{n-3}+2, c_{n-2} + c_n)) & - a_2(MC_1(c_1, c_2, ... ,c_{n-4}, c_{n-3}, c_{n-2} + c_n)) \\
& = - \frac{1}{2} ( c_{n-2} + c_n )  \quad, ... , \\
a_2(MC_1(c_1, c_2, ... , c_{n-4}, 0, c_{n-2} + c_n)) & - a_2(MC_1(c_1, c_2, ... , c_{n-4}, -2, c_{n-2} + c_n)) \\
& = - \frac{1}{2} ( c_{n-2} + c_n ) .
\end{align*}

Here, $MC_1((c_1, c_2, ... , c_{n-5}, c_{n-4}, 0, c_{n-2} + c_n)) \cong MC_1((c_1, c_2, ... , c_{n-5}, c_{n-4} + c_{n-2} + c_n))$.

Proceeding further, we have 
\begin{align*}
a_2(MC_1(c_1, c_2 + 2, c_3 + c_5 +... + c_{n-2} + c_n))
& - a_2(MC_1(c_1, c_2, c_3 + c_5 + ... + c_{n-2} + c_n)) \\
& = - \frac{1}{2} ( c_3 + c_5 + ... + c_{n-2} + c_n )
\quad, ... ,  \\
a_2(MC_1(c_1, 0, c_3 + c_5 + ... + c_{n-2} + c_n))
& - a_2(MC_1(c_1, -2, c_3 + c_5 + ... + c_{n-2} + c_n)) \\
& = - \frac{1}{2} ( c_3 + c_5 + ... + c_{n-2} + c_n ) .
\end{align*}

Here, $MC_1(c_1, 0, c_3 + c_5 + ... + c_{n-2} + c_n) \cong MC_1(c_1 + c_3 + c_5 + ... + c_{n-2} + c_n)$.

By summing up, we obtain 
\begin{align*}
& - a_2(MC_1(c_1, c_2, ... , c_n)) + a_2(MC_1(c_1 + c_3 + c_5 + ... + c_{n-2} + c_n))\\
& = \frac{1}{4} \{ c_n c_{n-1} + (c_{n-2} + c_n) c_{n-3} + ... + ( c_3 + ... + c_n ) c_2 \} \\
& = W^{(1)}.
\end{align*}

Thus,
\begin{align*}
a_2(MC_1(c_1, c_2, ... , c_n)) 
- a_2(MC_1(c_1 + c_3 + c_5 + ... + c_{n-2} + c_n))
= -W^{(1)}.
\end{align*}

By the same argument, we have
\begin{align*}
a_2(K) - a_2(K') = - \sum_{k=1}^{r} W^{(k)},
\end{align*}
where
\begin{align*}
K' & = P\big( \sum_{i=1}^{(n(1)+1)/2} c^{(1)}_{2i-1} ,  \sum_{i=1}^{(n(2)+1)/2}c^{(2)}_{2i-1}, ... ,
\sum_{i=1}^{(n(r)+1)/2}c^{(r)}_{2i-1} \big) \\
& = P(2X^{(1)}, 2X^{(2)}, ... , 2X^{(r)}).
\end{align*}

Here $K'$ is a positive pretzel knot of odd type.
By Proposition \ref{prop:preodd}, we obtain
\begin{align*}
a_2(K) & = - \sum_{k=1}^{r} W^{(k)} + a_2(K') \\
& = - \sum_{k=1}^{r} W^{(k)}
+ \sum_{1 \leq i < j}^{r} X^{(i)} X^{(j)} + \frac{1}{8} ( r-1 ).
\end{align*}

By the same way, we can deform $K$ into 
$K' = M\big( 0 ; ({\alpha_1}', {\beta_1}'), \ldots, ({\alpha}'_r, {\beta}'_r)\big)$, where
$L \big( T({\beta}'_k, {\alpha}'_k) \big)
= C\big( c^{(k)}_1 + c^{(k)}_3 + ...  + c^{(k)}_{n(k) - 2} + c^{(k)}_{n(k)} \big) 
= C\big(2X^{(k)}\big)$
for $1 \leq k \leq r$.

By Claim \ref{claim:ABCDEF}, we have 
\begin{align*}
d_G^{\Delta}(K, K') 
\leq  - \sum_{k=1}^{r} W^{(k)}.
\end{align*}

Here, 
\begin{align*}
K' & = P\big( \sum_{i=1}^{(n(1)+1)/2} c^{(1)}_{2i-1} ,  \sum_{i=1}^{(n(2)+1)/2}c^{(2)}_{2i-1}, ... ,
\sum_{i=1}^{(n(r)+1)/2}c^{(r)}_{2i-1} \big) \\
& = P(2X^{(1)}, 2X^{(2)}, ... , 2X^{(r)}).
\end{align*}

By Proposition \ref{prop:preodd}, we have 
\begin{align*}
u^{\Delta}(K') 
& = \sum_{1 \leq i < j}^{r} X^{(i)} X^{(j)} + \frac{1}{8} ( r-1 ).
\end{align*}

Therefore, we have 
\begin{align*}
u^{\Delta}(K) 
& \leq d_G^{\Delta}(K, K') + d_G^{\Delta}(K', O) \\
& = - \sum_{k=1}^{r} W^{(k)} 
+ \sum_{1 \leq i < j}^{r} X^{(i)} X^{(j)} + \frac{1}{8} ( r-1 )\\
& = a_2(K) = |a_2(K)|.
\end{align*}

By Proposition \ref{prop:a2}, we also have $u^{\Delta}(K) \geq |a_2(K)|$. 
Thus, we conclude that $u^{\Delta}(K)=|a_2(K)|$.
The proof is complete.
\end{proof}

\subsection{Proof of (2)} 
\label{subsec32}

\leavevmode

\begin{proof}
(2)
Let $k_+ = MC_1(c_1, c_2, ... , c_{n-2}, c_{n-1}, c_n)$ and $k_- = MC_1(c_1, c_2, ... , c_{n-2}, c_{n-1}-2, c_n)$. By Proposition \ref{prop:lk}, $a_2(MC_1(c_1, c_2, ... , c_{n-2}, c_{n-1}, c_n)) - a_2(MC_1(c_1, c_2, ... , c_{n-2}, c_{n-1}-2, c_n)) = lk(k_0) = \frac{1}{2} c_n$.

By repeating the  computation $\frac{1}{2}c_{n-1}$ times, we have
\begin{align*}
a_2(MC_1(c_1, c_2, ... , c_{n-2}, c_{n-1}, c_n)) & - a_2(MC_1(c_1, c_2, ... , c_{n-2}, c_{n-1}-2, c_n)) =  \frac{1}{2} c_n \quad, ... , \\
a_2(MC_1(c_1, c_2, ... , c_{n-2}, 2, c_n)) & - a_2(MC_1(c_1, c_2, ... , c_{n-2}, 0, c_n)) = \frac{1}{2} c_n .
\end{align*}

Here, $MC_1(c_1, c_2, ... , c_{n-3}, c_{n-2}, 0, c_n) \cong MC_1(c_1, c_2, ... ,c_{n-3}, c_{n-2} + c_n)$.

By continuing with another $\frac{1}{2} c_{n-3}$ steps, we have
\begin{align*}
a_2(MC_1(c_1, c_2, ... , c_{n-4}, c_{n-3}, c_{n-2} + c_n)) & - a_2(MC_1(c_1, c_2, ... ,c_{n-4}, c_{n-3}-2, c_{n-2} + c_n)) \\
& = \frac{1}{2} ( c_{n-2} + c_n )  \quad, ... , \\
a_2(MC_1(c_1, c_2, ... , c_{n-4}, 2, c_{n-2} + c_n)) & - a_2(MC_1(c_1, c_2, ... , c_{n-4}, 0, c_{n-2} + c_n)) \\
& = \frac{1}{2} ( c_{n-2} + c_n ) .
\end{align*}

Here, $MC_1((c_1, c_2, ... , c_{n-5}, c_{n-4}, 0, c_{n-2} + c_n)) \cong MC_1((c_1, c_2, ... , c_{n-5}, c_{n-4} + c_{n-2} + c_n))$.

Proceeding further, we have 
\begin{align*}
a_2(MC_1(c_1, c_2, c_3 + c_5 +... + c_{n-2} + c_n)) & - a_2(MC_1(c_1, c_2 - 2, c_3 + c_5 + ... + c_{n-2} + c_n)) \\
& = \frac{1}{2} ( c_3 + c_5 + ... + c_{n-2} + c_n )
\quad, ... ,  \\
a_2(MC_1(c_1, 2, c_3 + c_5 + ... + c_{n-2} + c_n)) & - a_2(MC_1(c_1, 0, c_3 + c_5 + ... + c_{n-2} + c_n)) \\
& = \frac{1}{2} ( c_3 + c_5 + ... + c_{n-2} + c_n ) .
\end{align*}

Here, $MC_1(c_1, 0, c_3 + c_5 + ... + c_{n-2} + c_n) \cong MC_1(c_1 + c_3 + c_5 + ... + c_{n-2} + c_n)$.

By summing up, we obtain 
\begin{align*}
& a_2(MC_1(c_1, c_2, ... , c_n)) - a_2(MC_1(c_1 + c_3 + c_5 + ... + c_{n-2} + c_n))\\
& = \frac{1}{4} \{ c_n c_{n-1} + (c_{n-2} + c_n) c_{n-3} + ... + ( c_3 + ... + c_n ) c_2 \} \\
& = W^{(1)}.
\end{align*}

Proceeding further, we have 
\begin{align*}
a_2( MC_1(c_1 + c_3 + c_5 + ... + c_{n-2} + c_n))
& - a_2( MC_1(c_1 + c_3 + c_5 + ... + c_{n-2} + c_n - 2)) \\
& = (\sum_{i=1}^{(n+1)/2}\frac{1}{2}c_{2i-1} - \frac{1}{2} )
+ \sum_{k=2}^{r} (\sum_{i=1}^{(n(k)+1)/2}\frac{1}{2}c^{(k)}_{2i-1}) , ... , \\
a_2(MC_1(2)) - a_2(MC_1(0))
& = \frac{1}{2}
+ \sum_{k=2}^{r} (\sum_{i=1}^{(n(k)+1)/2}\frac{1}{2}c^{(k)}_{2i-1}).
\end{align*}

By summing up, we obtain 
\begin{align*}
& a_2( MC_1(c_1 + c_3 + c_5 + ... + c_{n-2} + c_n))
 - a_2(MC_1(0)) \\
& = \frac{1}{2} \{ \sum_{i=1}^{(n+1)/2}\frac{1}{2}c_{2i-1} \}^2 
+  (\sum_{i=1}^{(n+1)/2}\frac{1}{2}c_{2i-1}) 
\{ \sum_{k=2}^{r} 
(\sum_{i=1}^{(n(k)+1)/2} \frac{1}{2} c^{(k)}_{2i-1}) \} 
\\
& = \frac{1}{2} {X^{(1)}}^2 + X^{(1)} \sum_{k=2}^{r} X^{(k)}.
\end{align*}

Here, $MC_1(0) \cong L \big( T(\beta_2 / \alpha_2) \big) \# ... \# L \big( T(\beta_r / \alpha_r) \big)$,
hence $a_2(MC_1(0)) = a_2(L \big( T(\beta_2 / \alpha_2) \big) ) + ... + a_2(L \big( T(\beta_r / \alpha_r) \big) )$.

By Proposition \ref{prop:oddnew}(2), we obtain
\begin{align*}
& a_2(MC_1(c_1, c_2, ... , c_n)) \\
& = W^{(1)}
+ a_2(MC_1(c_1 + c_3 + c_5 + ... + c_{n-2} + c_n)) \\
& =  W^{(1)} + \frac{1}{2} {X^{(1)}}^2 + X^{(1)} \sum_{k=2}^{r} X^{(k)} 
+ a_2(L \big( T(\beta_2 / \alpha_2) \big) ) + ... 
+ a_2(L \big( T(\beta_r / \alpha_r) \big) ) \\
& =  W^{(1)} + \frac{1}{2} {X^{(1)}}^2 + X^{(1)} \sum_{k=2}^{r} X^{(k)} 
+ \sum_{k=2}^{r} W^{(k)}
+ \sum_{k=2}^{r} \frac{1}{8} \{{(2X^{(k)})}^2 -1\} \\
& =  \sum_{k=1}^{r} W^{(k)} 
+ \frac{1}{2} \sum_{k=1}^{r} {{X^{(k)}}^2}
+ X^{(1)} \sum_{k=2}^{r} X^{(k)} 
- \frac{1}{8} ( r-1 ).
\end{align*}

By the same way, we can deform $K$ into 
$K' = M\big( 0 ; ({\alpha_1}', {\beta_1}'), \ldots, ({\alpha}'_r, {\beta}'_r)\big)$, where
$L \big( T({\beta}'_k, {\alpha}'_k) \big)
= C\big( c^{(k)}_1 + c^{(k)}_3 + ...  + c^{(k)}_{n(k) - 2} + c^{(k)}_{n(k)} \big) 
= C\big( 2X^{(k)}\big)$ 
for $1 \leq k \leq r$.

By Claim \ref{claim:ABCDEF}, we have 
\begin{align*}
d_G^{\Delta}(K, K') 
\leq \sum_{k=1}^{r} W^{(k)}.
\end{align*}

Here, 
\begin{align*}
K' & = P\big( \sum_{i=1}^{(n(1)+1)/2} c^{(1)}_{2i-1} ,  \sum_{i=1}^{(n(2)+1)/2}c^{(2)}_{2i-1}, ... ,
\sum_{i=1}^{(n(r)+1)/2}c^{(r)}_{2i-1} \big) \\
& = P(2X^{(1)}, 2X^{(2)}, ... , 2X^{(r)}),
\end{align*}
then $K'$ is a positive pretzel knot of even type.

By Proposition \ref{prop:preeven}(1), we have 
\begin{align*}
u^{\Delta}(K') 
= \frac{1}{2} \sum_{k=1}^{r} {X^{(k)}}^2
+ X^{(1)} \sum_{k=2}^{r} X^{(k)} 
- \frac{1}{8} (r-1).
\end{align*}

Therefore, we have 
\begin{align*}
u^{\Delta}(K) 
& \leq d_G^{\Delta}(K, K') + d_G^{\Delta}(K', O) \\
& = \sum_{k=1}^{r} W^{(k)} 
+ \frac{1}{2} \sum_{k=1}^{r} {X^{(k)}}^2
+ X^{(1)} \sum_{k=2}^{r} X^{(k)} 
- \frac{1}{8} (r-1)\\
& = a_2(K) = |a_2(K)|.
\end{align*}

By Proposition \ref{prop:a2}, we also have $u^{\Delta}(K) \geq |a_2(K)|$. 
Thus, we conclude that $u^{\Delta}(K)=|a_2(K)|$.
The proof is complete.
\end{proof}

\subsection{Proof of (3)} 
\label{subsec33}

\leavevmode

\begin{proof}
(3)
Let $k_+ = MC_1(c_1, c_2, ... , c_{n-2}, c_{n-1}, c_n)$ and $k_- = MC_1(c_1, c_2, ... , c_{n-2}, c_{n-1}-2, c_n)$. By Proposition \ref{prop:lk}, $a_2(MC_1(c_1, c_2, ... , c_{n-2}, c_{n-1}, c_n)) - a_2(MC_1(c_1, c_2, ... , c_{n-2}, c_{n-1}-2, c_n)) = lk(k_0) = - \frac{1}{2} c_n$.

By repeating the  computation $\frac{1}{2}c_{n-1}$ times, we have
\begin{align*}
a_2(MC_1(c_1, c_2, ... , c_{n-2}, c_{n-1}, c_n)) 
& - a_2(MC_1(c_1, c_2, ... , c_{n-2}, c_{n-1}-2, c_n))
=  - \frac{1}{2} c_n \quad, ... , \\
a_2(MC_1(c_1, c_2, ... , c_{n-2}, 2, c_n)) 
& - a_2(MC_1(c_1, c_2, ... , c_{n-2}, 0, c_n))
= - \frac{1}{2} c_n .
\end{align*}

Here, $MC_1(c_1, c_2, ... , c_{n-3}, c_{n-2}, 0, c_n) \cong MC_1(c_1, c_2, ... ,c_{n-3}, c_{n-2} + c_n)$.

By continuing with another $\frac{1}{2} c_{n-3}$ steps, we have
\begin{align*}
a_2(MC_1(c_1, c_2, ... , c_{n-4}, c_{n-3}, c_{n-2} + c_n)) & - a_2(MC_1(c_1, c_2, ... ,c_{n-4}, c_{n-3}-2, c_{n-2} + c_n)) \\
& = - \frac{1}{2} ( c_{n-2} + c_n )  \quad, ... , \\
a_2(MC_1(c_1, c_2, ... , c_{n-4}, 2, c_{n-2} + c_n)) 
& - a_2(MC_1(c_1, c_2, ... , c_{n-4}, 0, c_{n-2} + c_n)) \\
& = - \frac{1}{2} ( c_{n-2} + c_n ) .
\end{align*}

Here, $MC_1((c_1, c_2, ... , c_{n-5}, c_{n-4}, 0, c_{n-2} + c_n)) \cong MC_1((c_1, c_2, ... , c_{n-5}, c_{n-4} + c_{n-2} + c_n))$.

Proceeding further, we have 
\begin{align*}
a_2(MC_1(c_1, c_2, c_3 + c_5 +... + c_{n-2} + c_n)) 
& - a_2(MC_1(c_1, c_2 - 2, c_3 + c_5 + ... + c_{n-2} + c_n)) \\
& = - \frac{1}{2} ( c_3 + c_5 + ... + c_{n-2} + c_n )
\quad, ... ,  \\
a_2(MC_1(c_1, 2, c_3 + c_5 + ... + c_{n-2} + c_n)) & - a_2(MC_1(c_1, 0, c_3 + c_5 + ... + c_{n-2} + c_n)) \\
& = - \frac{1}{2} ( c_3 + c_5 + ... + c_{n-2} + c_n ) .
\end{align*}

Here, $MC_1(c_1, 0, c_3 + c_5 + ... + c_{n-2} + c_n) \cong MC_1(c_1 + c_3 + c_5 + ... + c_{n-2} + c_n)$.

By summing up, we obtain 
\begin{align*}
& a_2(MC_1(c_1, c_2, ... , c_n)) - a_2(MC_1(c_1 + c_3 + c_5 + ... + c_{n-2} + c_n))\\
& = - \frac{1}{4} \{ c_n c_{n-1} + (c_{n-2} + c_n) c_{n-3} + ... + ( c_3 + ... + c_n ) c_2 \} \\
& = - W^{(1)}.
\end{align*}

Proceeding further, we have 
\begin{align*}
a_2( MC_1(c_1 + c_3 + c_5 + ... + c_{n-2} + c_n))
& - a_2( MC_1(c_1 + c_3 + c_5 + ... + c_{n-2} + c_n + 2)) \\
& = \sum_{k=2}^{r} (\sum_{i=1}^{(n(k)+1)/2}\frac{1}{2}c^{(k)}_{2i-1}) , ... , \\
a_2(MC_1(-2)) - a_2(MC_1(0))
& = \sum_{k=2}^{r} (\sum_{i=1}^{(n(k)+1)/2}\frac{1}{2}c^{(k)}_{2i-1}).
\end{align*}

By summing up, we obtain 
\begin{align*}
& a_2( MC_1(c_1 + c_3 + c_5 + ... + c_{n-2} + c_n))
 - a_2(MC_1(0)) \\
& =  ( - \sum_{i=1}^{(n+1)/2}\frac{1}{2}c_{2i-1})
\{ \sum_{k=2}^{r} 
(\sum_{i=1}^{(n(k)+1)/2} \frac{1}{2} c^{(k)}_{2i-1}) \}
 \\
& = - X^{(1)} \sum_{k=2}^{r} X^{(k)}.
\end{align*}

Here, $MC_1(0) \cong L \big( T(\beta_2 / \alpha_2) \big) \# ... \# L \big( T(\beta_r / \alpha_r) \big)$,
hence $a_2(MC_1(0)) = a_2(L \big( T(\beta_2 / \alpha_2) \big) ) + ... + a_2(L \big( T(\beta_r / \alpha_r) \big) )$.

By Proposition \ref{prop:oddnew}(2), we obtain
\begin{align*}
& a_2(MC_1(c_1, c_2, ... , c_n)) \\
& = - W^{(1)} 
+ a_2(MC_1(c_1 + c_3 + c_5 + ... + c_{n-2} + c_n)) \\
& =  - W^{(1)} - X^{(1)} \sum_{k=2}^{r} X^{(k)} 
+ a_2(L \big( T(\beta_2 / \alpha_2) \big) ) + ... 
+ a_2(L \big( T(\beta_r / \alpha_r) \big) ) \\
& =  - W^{(1)} - X^{(1)} \sum_{k=2}^{r} X^{(k)} 
+ \sum_{k=2}^{r} W^{(k)}
+ \sum_{k=2}^{r} \frac{1}{8} \{{(2X^{(k)})}^2 -1\} \\
& = - W^{(1)} + \sum_{k=2}^{r} W^{(k)}
+ \frac{1}{2} \sum_{k=2}^{r} {{X^{(k)}}^2}
- X^{(1)} \sum_{k=2}^{r} X^{(k)}
- \frac{1}{8} ( r-1 ).
\end{align*}

By the same way, we can deform $K$ into 
$K' = M\big( 0 ; ({\alpha_1}', {\beta_1}'), \ldots, ({\alpha}'_r, {\beta}'_r)\big)$, where
$L \big( T({\beta}'_k, {\alpha}'_k) \big)
= C\big( c^{(k)}_1 + c^{(k)}_3 + ...  + c^{(k)}_{n(k) - 2} + c^{(k)}_{n(k)} \big) 
= C \big( 2X^{(k)} \big)$ 
for $1 \leq k \leq r$.

By Claim \ref{claim:ABCDEF}, we have 
\begin{align*}
d_G^{\Delta}(K, K') 
& \leq  - W^{(1)} + \sum_{k=2}^{r} W^{(k)}.
\end{align*}

Here, 
\begin{align*}
K' & = P\big( \sum_{i=1}^{(n(1)+1)/2} c^{(1)}_{2i-1} ,  \sum_{i=1}^{(n(2)+1)/2}c^{(2)}_{2i-1}, ... ,
\sum_{i=1}^{(n(r)+1)/2}c^{(r)}_{2i-1} \big) \\
& = P(2X^{(1)}, 2X^{(2)}, ... , 2X^{(r)}),
\end{align*}
then $K'$ is a positive pretzel knot of even type.

By Proposition \ref{prop:preeven}(2), we have 
\begin{align*}
u^{\Delta}(K') 
= \frac{1}{2} \sum_{k=2}^{r} {X^{(k)}}^2
- X^{(1)} \sum_{k=2}^{r} X^{(k)} 
- \frac{1}{8} (r-1).
\end{align*}

Therefore, we have 
\begin{align*}
u^{\Delta}(K) 
& \leq d_G^{\Delta}(K, K') + d_G^{\Delta}(K', O) \\
& = -W^{(1)} + \sum_{k=2}^{r} W^{(k)} 
+ \frac{1}{2} \sum_{k=2}^{r} {X^{(k)}}^2
- X^{(1)} \sum_{k=2}^{r} X^{(k)} 
- \frac{1}{8} (r-1)\\
& = a_2(K) = |a_2(K)|.
\end{align*}

By Proposition \ref{prop:a2}, we also have $u^{\Delta}(K) \geq |a_2(K)|$. 
Thus, we conclude that $u^{\Delta}(K)=|a_2(K)|$.
The proof is complete.
\end{proof}

\subsection{Proof of (4)} 
\label{subsec34}

\leavevmode

\begin{proof}
(4)
Let $k_+ = MC_1(c_1, c_2, ... , c_{n-2}, c_{n-1}, c_n)$ and $k_- = MC_1(c_1, c_2, ... , c_{n-2}, c_{n-1}-2, c_n)$. By Proposition \ref{prop:lk}, $a_2(MC_1(c_1, c_2, ... , c_{n-2}, c_{n-1}, c_n)) - a_2(MC_1(c_1, c_2, ... , c_{n-2}, c_{n-1}-2, c_n)) = lk(k_0) = - \frac{1}{2} c_n$.

By repeating the  computation $\frac{1}{2}c_{n-1}$ times, we have
\begin{align*}
a_2(MC_1(c_1, c_2, ... , c_{n-2}, c_{n-1}, c_n)) & - a_2(MC_1(c_1, c_2, ... , c_{n-2}, c_{n-1}-2, c_n)) = - \frac{1}{2} c_n \quad, ... , \\
a_2(MC_1(c_1, c_2, ... , c_{n-2}, 2, c_n)) & - a_2(MC_1(c_1, c_2, ... , c_{n-2}, 0, c_n)) = - \frac{1}{2} c_n .
\end{align*}

Here, $MC_1(c_1, c_2, ... , c_{n-3}, c_{n-2}, 0, c_n) \cong MC_1(c_1, c_2, ... ,c_{n-3}, c_{n-2} + c_n)$.

By continuing with another $\frac{1}{2} c_{n-3}$ steps, we have
\begin{align*}
a_2(MC_1(c_1, c_2, ... , c_{n-4}, c_{n-3}, c_{n-2} + c_n)) & - a_2(MC_1(c_1, c_2, ... ,c_{n-4}, c_{n-3}-2, c_{n-2} + c_n)) \\
& = - \frac{1}{2} ( c_{n-2} + c_n )  \quad, ... , \\
a_2(MC_1(c_1, c_2, ... , c_{n-4}, 2, c_{n-2} + c_n)) & - a_2(MC_1(c_1, c_2, ... , c_{n-4}, 0, c_{n-2} + c_n)) \\
& = - \frac{1}{2} ( c_{n-2} + c_n ) .
\end{align*}

Here, $MC_1((c_1, c_2, ... , c_{n-5}, c_{n-4}, 0, c_{n-2} + c_n)) \cong MC_1((c_1, c_2, ... , c_{n-5}, c_{n-4} + c_{n-2} + c_n))$.

Proceeding further, we have 
\begin{align*}
a_2(MC_1(c_1, c_2, c_3 + c_5 +... + c_{n-2} + c_n)) & - a_2(MC_1(c_1, c_2 - 2, c_3 + c_5 + ... + c_{n-2} + c_n)) \\
& = - \frac{1}{2} ( c_3 + c_5 + ... + c_{n-2} + c_n )
\quad, ... ,  \\
a_2(MC_1(c_1, 2, c_3 + c_5 + ... + c_{n-2} + c_n)) & - a_2(MC_1(c_1, 0, c_3 + c_5 + ... + c_{n-2} + c_n)) \\
& = - \frac{1}{2} ( c_3 + c_5 + ... + c_{n-2} + c_n ) .
\end{align*}

Here, $MC_1(c_1, 0, c_3 + c_5 + ... + c_{n-2} + c_n) \cong MC_1(c_1 + c_3 + c_5 + ... + c_{n-2} + c_n)$.

By summing up, we obtain 
\begin{align*}
& a_2(MC_1(c_1, c_2, ... , c_n)) - a_2(MC_1(c_1 + c_3 + c_5 + ... + c_{n-2} + c_n))\\
& = - \frac{1}{4} \{ c_n c_{n-1} + (c_{n-2} + c_n) c_{n-3} + ... + ( c_3 + ... + c_n ) c_2 \} \\
& = - W^{(1)}.
\end{align*}

Proceeding further, we have 
\begin{align*}
a_2( MC_1(c_1 + c_3 + c_5 + ... + c_{n-2} + c_n - 2))
& - a_2( MC_1(c_1 + c_3 + c_5 + ... + c_{n-2} + c_n)) \\
& = \sum_{k=2}^{r} (\sum_{i=1}^{n(k)/2}\frac{1}{2}c^{(k)}_{2i}) , ... , \\
a_2(MC_1(0)) - a_2(MC_1(2)) & = \sum_{k=2}^{r} (\sum_{i=1}^{n(k)/2}\frac{1}{2}c^{(k)}_{2i}).
\end{align*}

By summing up, we obtain 
\begin{align*}
a_2(MC_1(0)) - a_2( MC_1(c_1 + c_3 + c_5 + ... + c_{n-2} + c_n)) 
& = (\sum_{i=1}^{(n+1)/2}\frac{1}{2}c_{2i-1})
\{ \sum_{k=2}^{r} (\sum_{i=1}^{n(k)/2}\frac{1}{2}c^{(k)}_{2i}) \} \\
& = X^{(1)} \sum_{k=2}^{r} Y^{(k)}.
\end{align*}

Here, $MC_1(0) \cong L \big( T(\beta_2 / \alpha_2) \big) \# ... \# L \big( T(\beta_r / \alpha_r) \big)$,
hence $a_2(MC_1(0)) = a_2(L \big( T(\beta_2 / \alpha_2) \big) ) + ... + a_2(L \big( T(\beta_r / \alpha_r) \big) )$.

By Proposition \ref{prop:even}, we obtain
\begin{align*}
& a_2(MC_1(c_1, c_2, ... , c_n)) \\
& = - W^{(1)}
+ a_2(MC_1(c_1 + c_3 + c_5 + ... + c_{n-2} + c_n)) \\
& = - W^{(1)}
- X^{(1)} \sum_{k=2}^{r} Y^{(k)}
+ a_2(L \big( T(\beta_2 / \alpha_2) \big) ) + ... 
+ a_2(L \big( T(\beta_r / \alpha_r) \big) ) \\
& = - W^{(1)}
- \sum_{k=2}^{r} V^{(k)}
- X^{(1)} \sum_{k=2}^{r} Y^{(k)}.
\end{align*}

By the same way, we can deform $K$ into 
$K' = M\big( 0 ; ({\alpha_1}', {\beta_1}'), \ldots, ({\alpha}'_r, {\beta}'_r)\big)$, where
$L \big( T({\beta}'_1, {\alpha}'_1) \big)
= C\big(c^{(1)}_1 + c^{(1)}_3 + ...  + c^{(1)}_{n(1) - 2} + c^{(1)}_{n(1)} \big)
= C \big( 2X^{(1)} \big)$, and
$L \big( T({\beta}'_k, {\alpha}'_k) \big)
= C\big(0, c^{(k)}_2 + c^{(k)}_4 + ...  + c^{(k)}_{n(k) - 2} + c^{(k)}_{n(k)} \big) 
= C \big(0, 2Y^{(k)} \big)$
for $2 \leq k \leq r$.

By Claim \ref{claim:ABCDEF}, we have 
\begin{align*}
d_G^{\Delta}(K, K') 
\leq W^{(1)} + \sum_{k=2}^{r} V^{(k)}.
\end{align*}

Here, 
\begin{align*}
K'= C\big( \sum_{i=1}^{(n(1)+1)/2} c^{(1)}_{2i-1} , \sum_{k=2}^{r} ( \sum_{i=1}^{n(k)/2}c^{(k)}_{2i}) \big)
= C(2 X^{(1)}, 2 \sum_{k=2}^{r} Y^{(k)}).
\end{align*}

By Proposition \ref{prop:even}, we have 
\begin{align*}
u^{\Delta}(K') 
= X^{(1)} \sum_{k=2}^{r} Y^{(k)}.
\end{align*}

Therefore, we have 
\begin{align*}
u^{\Delta}(K) 
& \leq d_G^{\Delta}(K, K') + d_G^{\Delta}(K', O) \\
& \leq W^{(1)} + \sum_{k=2}^{r} V^{(k)} 
+ X^{(1)} \sum_{k=2}^{r} Y^{(k)}\\
& = -a_2(K) = |a_2(K)|.
\end{align*}

By Proposition \ref{prop:a2}, we also have $u^{\Delta}(K) \geq |a_2(K)|$. 
Thus, we conclude that $u^{\Delta}(K)=|a_2(K)|$.
The proof is complete.
\end{proof}

\subsection{Proof of (5)} 
\label{subsec35}

\leavevmode

\begin{proof}
(5)
Let $k_+ = MC_1(c_1, c_2, ... , c_{n-2}, c_{n-1}+2, c_n)$ and $K = k_- = MC_1(c_1, c_2, ... , c_{n-2}, c_{n-1}, c_n)$. 
By Proposition \ref{prop:lk}, 
$a_2(MC_1(c_1, c_2, ... , c_{n-2}, c_{n-1}+2, c_n)) - a_2(MC_1(c_1, c_2, ... , c_{n-2}, c_{n-1}, c_n)) = lk(k_0) = - \frac{1}{2} c_n$.

By repeating the computation $- \frac{1}{2}c_{n-1}$ times, we have
\begin{align*}
a_2(MC_1(c_1, c_2, ... , c_{n-2}, c_{n-1}+2, c_n))
& - a_2(MC_1(c_1, c_2, ... , c_{n-2}, c_{n-1}, c_n)) 
=  - \frac{1}{2} c_n \quad, ... , \\
a_2(MC_1(c_1, c_2, ... , c_{n-2}, 0, c_n)) 
& - a_2(MC_1(c_1, c_2, ... , c_{n-2}, -2, c_n))
= - \frac{1}{2} c_n .
\end{align*}

Here, $MC_1(c_1, c_2, ... , c_{n-3}, c_{n-2}, 0, c_n) \cong MC_1(c_1, c_2, ... ,c_{n-3}, c_{n-2} + c_n)$.

By continuing with another $- \frac{1}{2} c_{n-3}$ steps, we have
\begin{align*}
a_2(MC_1(c_1, c_2, ... , c_{n-4}, c_{n-3}+2, c_{n-2} + c_n)) & - a_2(MC_1(c_1, c_2, ... ,c_{n-4}, c_{n-3}, c_{n-2} + c_n)) \\
& = - \frac{1}{2} ( c_{n-2} + c_n )  \quad, ... , \\
a_2(MC_1(c_1, c_2, ... , c_{n-4}, 0, c_{n-2} + c_n)) & - a_2(MC_1(c_1, c_2, ... , c_{n-4}, -2, c_{n-2} + c_n)) \\
& = - \frac{1}{2} ( c_{n-2} + c_n ) .
\end{align*}

Here, $MC_1((c_1, c_2, ... , c_{n-5}, c_{n-4}, 0, c_{n-2} + c_n)) \cong MC_1((c_1, c_2, ... , c_{n-5}, c_{n-4} + c_{n-2} + c_n))$.

Proceeding further, we have 
\begin{align*}
a_2(MC_1(c_1, c_2 + 2, c_3 + c_5 +... + c_{n-2} + c_n))
& - a_2(MC_1(c_1, c_2, c_3 + c_5 + ... + c_{n-2} + c_n)) \\
& = - \frac{1}{2} ( c_3 + c_5 + ... + c_{n-2} + c_n )
\quad, ... ,  \\
a_2(MC_1(c_1, 0, c_3 + c_5 + ... + c_{n-2} + c_n))
& - a_2(MC_1(c_1, -2, c_3 + c_5 + ... + c_{n-2} + c_n)) \\
& = - \frac{1}{2} ( c_3 + c_5 + ... + c_{n-2} + c_n ) .
\end{align*}

Here, $MC_1(c_1, 0, c_3 + c_5 + ... + c_{n-2} + c_n) \cong MC_1(c_1 + c_3 + c_5 + ... + c_{n-2} + c_n)$.

By summing up, we obtain 
\begin{align*}
& - a_2(MC_1(c_1, c_2, ... , c_n)) + a_2(MC_1(c_1 + c_3 + c_5 + ... + c_{n-2} + c_n))\\
& = \frac{1}{4} \{ c_n c_{n-1} + (c_{n-2} + c_n) c_{n-3} + ... + ( c_3 + ... + c_n ) c_2 \} \\
& = W^{(1)}.
\end{align*}

Thus,
\begin{align*}
a_2(MC_1(c_1, c_2, ... , c_n)) 
- a_2(MC_1(c_1 + c_3 + c_5 + ... + c_{n-2} + c_n))
= -W^{(1)}.
\end{align*}

By the same argument, we have
\begin{align*}
a_2(K) - a_2(K') = - W^{(1)} + \sum_{k=2}^{r} V^{(k)},
\end{align*}
where
\begin{align*}
K' & = C\big( \sum_{i=1}^{(n(1)+1)/2} c^{(1)}_{2i-1} ,  \sum_{i=1}^{n(2)/2}c^{(2)}_{2i} + ... +
\sum_{i=1}^{n(r)/2}c^{(r)}_{2i} \big) \\
& = C(2X^{(1)}, 2\sum_{k=2}^{r}Y^{(k)}).
\end{align*}

Here $K'$ is a two-bridge knot, where $2X^{(1)}$ is a positive odd integer and $2\sum_{k=2}^{r}Y^{(k)}$ is a positive even integer.
By Proposition \ref{prop:oddnew}(1), we obtain
\begin{align*}
a_2(K) & = - W^{(1)} + \sum_{k=2}^{r} V^{(k)} + a_2(K') \\
& = - W^{(1)} + \sum_{k=2}^{r} V^{(k)} 
+ X^{(1)} \sum_{k=2}^{r}Y^{(k)}
+ \frac{1}{8}\{2X^{(1)}\}^2 \\
& = - W^{(1)} + \sum_{k=2}^{r} V^{(k)} 
+ \frac{1}{2} {X^{(1)}}^2
+ X^{(1)} \sum_{k=2}^{r}Y^{(k)}.
\end{align*}

By the same way, we can deform $K$ into 
$K' = M\big( 0 ; ({\alpha_1}', {\beta_1}'), \ldots, ({\alpha}'_r, {\beta}'_r)\big)$, where
$L \big( T({\beta}'_1, {\alpha}'_1) \big)
= C\big( c^{(1)}_1 + c^{(1)}_3 + ...  + c^{(1)}_{n(1) - 2} + c^{(1)}_{n(1)} \big)
= C(2X^{(1)})$,
and $L \big( T({\beta}'_k, {\alpha}'_k) \big)
= C\big( 0, c^{(k)}_2 + c^{(k)}_4 + ...  + c^{(k)}_{n(k) - 2} + c^{(k)}_{n(k)} \big) 
= C\big(0, 2Y^{(k)}\big)$
for $2 \leq k \leq r$.

By Claim \ref{claim:ABCDEF}, we have 
\begin{align*}
d_G^{\Delta}(K, K') 
\leq  -W^{(1)} + \sum_{k=2}^{r} V^{(k)}.
\end{align*}

Here, 
\begin{align*}
K' & = C\big( \sum_{i=1}^{(n(1)+1)/2} c^{(1)}_{2i-1} ,  \sum_{i=1}^{n(2)/2}c^{(2)}_{2i} + ... +
\sum_{i=1}^{n(r)/2}c^{(r)}_{2i} \big) \\
& = C(2X^{(1)}, 2\sum_{k=2}^{r}Y^{(k)}).
\end{align*}

By Proposition \ref{prop:oddnew}(1), we have 
\begin{align*}
u^{\Delta}(K') 
& = X^{(1)} \sum_{k=2}^{r}Y^{(k)}
+ \frac{1}{8}\{2X^{(1)}\}^2 \\
& = \frac{1}{2} {X^{(1)}}^2
+ X^{(1)} \sum_{k=2}^{r}Y^{(k)}.
\end{align*}

Therefore, we have 
\begin{align*}
u^{\Delta}(K) 
& \leq d_G^{\Delta}(K, K') + d_G^{\Delta}(K', O) \\
& = - W^{(1)} + \sum_{k=2}^{r} V^{(k)} 
+ \frac{1}{2} {X^{(1)}}^2
+ X^{(1)} \sum_{k=2}^{r}Y^{(k)} \\
& = a_2(K) = |a_2(K)|.
\end{align*}

By Proposition \ref{prop:a2}, we also have $u^{\Delta}(K) \geq |a_2(K)|$. 
Thus, we conclude that $u^{\Delta}(K)=|a_2(K)|$.
The proof is complete.
\end{proof}

\subsection{Proof of (6)} 
\label{subsec36}

\leavevmode

\begin{proof}
(6)
Let $k_+ = MC_1(c_1, c_2, ... , c_{n-2}, c_{n-1}, c_n)$ and $k_- = MC_1(c_1, c_2, ... , c_{n-2}, c_{n-1}-2, c_n)$. By Proposition \ref{prop:lk}, $a_2(MC_1(c_1, c_2, ... , c_{n-2}, c_{n-1}, c_n)) - a_2(MC_1(c_1, c_2, ... , c_{n-2}, c_{n-1}-2, c_n)) = lk(k_0) = - \frac{1}{2} c_n$.

By repeating the  computation $\frac{1}{2}c_{n-1}$ times, we have
\begin{align*}
a_2(MC_1(c_1, c_2, ... , c_{n-2}, c_{n-1}, c_n)) 
& - a_2(MC_1(c_1, c_2, ... , c_{n-2}, c_{n-1}-2, c_n))
=  - \frac{1}{2} c_n \quad, ... , \\
a_2(MC_1(c_1, c_2, ... , c_{n-2}, 2, c_n)) 
& - a_2(MC_1(c_1, c_2, ... , c_{n-2}, 0, c_n))
= - \frac{1}{2} c_n .
\end{align*}

Here, $MC_1(c_1, c_2, ... , c_{n-3}, c_{n-2}, 0, c_n) \cong MC_1(c_1, c_2, ... ,c_{n-3}, c_{n-2} + c_n)$.

By continuing with another $\frac{1}{2} c_{n-3}$ steps, we have
\begin{align*}
a_2(MC_1(c_1, c_2, ... , c_{n-4}, c_{n-3}, c_{n-2} + c_n)) & - a_2(MC_1(c_1, c_2, ... ,c_{n-4}, c_{n-3}-2, c_{n-2} + c_n)) \\
& = - \frac{1}{2} ( c_{n-2} + c_n )  \quad, ... , \\
a_2(MC_1(c_1, c_2, ... , c_{n-4}, 2, c_{n-2} + c_n)) 
& - a_2(MC_1(c_1, c_2, ... , c_{n-4}, 0, c_{n-2} + c_n)) \\
& = - \frac{1}{2} ( c_{n-2} + c_n ) .
\end{align*}

Here, $MC_1((c_1, c_2, ... , c_{n-5}, c_{n-4}, 0, c_{n-2} + c_n)) \cong MC_1((c_1, c_2, ... , c_{n-5}, c_{n-4} + c_{n-2} + c_n))$.

Proceeding further, we have 
\begin{align*}
a_2(MC_1(c_1, c_2, c_3 + c_5 +... + c_{n-2} + c_n)) 
& - a_2(MC_1(c_1, c_2 - 2, c_3 + c_5 + ... + c_{n-2} + c_n)) \\
& = - \frac{1}{2} ( c_3 + c_5 + ... + c_{n-2} + c_n )
\quad, ... ,  \\
a_2(MC_1(c_1, 2, c_3 + c_5 + ... + c_{n-2} + c_n)) & - a_2(MC_1(c_1, 0, c_3 + c_5 + ... + c_{n-2} + c_n)) \\
& = - \frac{1}{2} ( c_3 + c_5 + ... + c_{n-2} + c_n ) .
\end{align*}

Here, $MC_1(c_1, 0, c_3 + c_5 + ... + c_{n-2} + c_n) \cong MC_1(c_1 + c_3 + c_5 + ... + c_{n-2} + c_n)$.

By summing up, we obtain 
\begin{align*}
& a_2(MC_1(c_1, c_2, ... , c_n)) - a_2(MC_1(c_1 + c_3 + c_5 + ... + c_{n-2} + c_n))\\
& = - \frac{1}{4} \{ c_n c_{n-1} + (c_{n-2} + c_n) c_{n-3} + ... + ( c_3 + ... + c_n ) c_2 \} \\
& = - W^{(1)}.
\end{align*}

Proceeding further, we have 
\begin{align*}
a_2( MC_1(c_1 + c_3 + c_5 + ... + c_{n-2} + c_n))
& - a_2( MC_1(c_1 + c_3 + c_5 + ... + c_{n-2} + c_n + 2)) \\
& = \sum_{k=2}^{r} \{\sum_{i=1}^{n(k)/2}
(-\frac{1}{2}c^{(k)}_{2i})\} , ... , \\
a_2(MC_1(-2)) - a_2(MC_1(0))
& = \sum_{k=2}^{r} \{\sum_{i=1}^{n(k)/2}
(-\frac{1}{2}c^{(k)}_{2i})\}.
\end{align*}

By summing up, we obtain 
\begin{align*}
& a_2( MC_1(c_1 + c_3 + c_5 + ... + c_{n-2} + c_n))
 - a_2(MC_1(0)) \\
& =  ( - \sum_{i=1}^{(n+1)/2}\frac{1}{2}c_{2i-1})
[ \sum_{k=2}^{r} 
\{\sum_{i=1}^{n(k)/2}(- \frac{1}{2} c^{(k)}_{2i}) \} ]
 \\
& = X^{(1)} \sum_{k=2}^{r} Y^{(k)}.
\end{align*}

Here, $MC_1(0) \cong L \big( T(\beta_2 / \alpha_2) \big) \# ... \# L \big( T(\beta_r / \alpha_r) \big)$,
hence $a_2(MC_1(0)) = a_2(L \big( T(\beta_2 / \alpha_2) \big) ) + ... + a_2(L \big( T(\beta_r / \alpha_r) \big) )$.

By Proposition \ref{prop:oddnew}(1), we obtain
\begin{align*}
& a_2(MC_1(c_1, c_2, ... , c_n)) \\
& = - W^{(1)} 
+ a_2(MC_1(c_1 + c_3 + c_5 + ... + c_{n-2} + c_n)) \\
& =  - W^{(1)}
+ X^{(1)} \sum_{k=2}^{r} Y^{(k)} 
+ a_2(L \big( T(\beta_2 / \alpha_2) \big) ) + ... 
+ a_2(L \big( T(\beta_r / \alpha_r) \big) ) \\
& =  - W^{(1)}
+ X^{(1)} \sum_{k=2}^{r} Y^{(k)} 
+ \sum_{k=2}^{r} V^{(k)}
+ \sum_{k=2}^{r} \frac{1}{8} {(2Y^{(k)})}^2 \\
& = - W^{(1)} + \sum_{k=2}^{r} V^{(k)}
+ \frac{1}{2} \sum_{k=2}^{r} {{Y^{(k)}}^2}
+ X^{(1)} \sum_{k=2}^{r} Y^{(k)}.
\end{align*}

By the same way, we can deform $K$ into 
$K' = MC_1(0) \cong L \big( T(\beta_2 / \alpha_2) \big) \# ... \# L \big( T(\beta_r / \alpha_r) \big)$.

By Claim \ref{claim:ABCDEF}, we have 
\begin{align*}
d_G^{\Delta}(K, K') 
& \leq d_G^{\Delta} \big(K, MC(c_1 + c_3 + ... + c_n) \big) 
+ d_G^{\Delta} \big(MC(c_1 + c_3 + ... + c_n), K' \big) \\ 
& \leq - W^{(1)}
+ \{\sum_{i=1}^{(n+1)/2}\frac{1}{2}(-c_{2i-1})\}
[ \sum_{k=2}^{r} 
\{\sum_{i=1}^{n(k)/2}(- \frac{1}{2} c^{(k)}_{2i}) \} ] \\
& = - W^{(1)}
+ X^{(1)} \sum_{k=2}^{r} Y^{(k)}.
\end{align*}

By Proposition \ref{prop:oddnew}(1), we have 
\begin{align*}
u^{\Delta}(K') 
= \sum_{k=2}^{r} V^{(k)}
+ \sum_{k=2}^{r} \frac{1}{8} {(2Y^{(k)})}^2.
\end{align*}

Therefore, we have 
\begin{align*}
u^{\Delta}(K) 
& \leq d_G^{\Delta}(K, K') + d_G^{\Delta}(K', O) \\
& = - W^{(1)} + \sum_{k=2}^{r} V^{(k)}
+ \frac{1}{2} \sum_{k=2}^{r} {{Y^{(k)}}^2}
+ X^{(1)} \sum_{k=2}^{r} Y^{(k)}\\
& = a_2(K) = |a_2(K)|.
\end{align*}

By Proposition \ref{prop:a2}, we also have $u^{\Delta}(K) \geq |a_2(K)|$. 
Thus, we conclude that $u^{\Delta}(K)=|a_2(K)|$.
The proof is complete.
\end{proof}

\section{Remarks on Theorem \ref{thm:main}}\label{sec4}



We make the following remarks on Theorem \ref{thm:main}.

\begin{remark}
The correspondence between the parts of Theorem \ref{thm:main}
and the previously established propositions is as follows.

\begin{description}
\item[(1)] Proposition \ref{prop:preodd}.
\item[(2)] Proposition \ref{prop:preeven}(1) and Proposition \ref{prop:oddnew}(2).
\item[(3)] Proposition \ref{prop:preeven}(2) and Proposition \ref{prop:oddnew}(2).
\item[(4)] Proposition \ref{prop:even}.
\item[(5)] Proposition \ref{prop:oddnew}(1).
\item[(6)] Proposition \ref{prop:oddnew}(1).
\end{description}

Consequently, every positive pretzel knot in Propositions \ref{prop:preodd} and \ref{prop:preeven} belongs to one of the
families described in Parts (1)--(3), and every two-bridge knot in
Propositions \ref{prop:even} and \ref{prop:oddnew} belongs to one of the families described in Parts (2)--(6).
\end{remark}

\begin{remark}\label{remark:thm11}
In the setting of Theorem \ref{thm:main}, 
let $K' = M\big(0; (\alpha'_1,\beta'_1), \dots, (\alpha'_r,\beta'_r)\big)$
be a Montesinos knot, where $(\alpha'_k,\beta'_k)$ is defined in Claim \ref{claim:ABCDEF} for $1 \leq k \leq r$.
Then we have
\begin{align*}
u^{\Delta}(K)
= d_G^{\Delta}(K, K') + u^{\Delta}(K'). 
\end{align*}
\end{remark}



\section{Properties of the $\Delta$-Unknotting Number for Montesinos Knots}\label{sec5}



It is known that there exist Montesinos knots for which $u^{\Delta}(K) \ne |a_2(K)|$.
For example,
$9_{25} \, \big( u^{\Delta}(9_{25}) = 2,\, a_2(9_{25}) = 0 \big)$, and 
$9_{44} \, \big( u^{\Delta}(9_{44}) = 2,\, a_2(9_{44}) = 0 \big)$.
The knot $9_{25}$ is alternating, whereas  $9_{44}$ is almost alternating. 

On the other hand, there exist Montesinos knots for which $u^{\Delta}(K) = |a_2(K)|$. For example,
$8_{10} \, \big( u^{\Delta}(8_{10}) = a_2(8_{10}) = 3 \big)$, and
$8_{20} \, \big( u^{\Delta}(8_{20}) = a_2(8_{20}) = 2 \big)$.
Here, $8_{10}$ is alternating, whereas $8_{20}$ is almost alternating.

Thus, the equality $u^{\Delta}(K) = |a_2(K)|$ is independent of whether the knot is alternating or almost alternating.
\begin{figure}[htbp]
    \centering
    \includegraphics[width=0.8\linewidth]{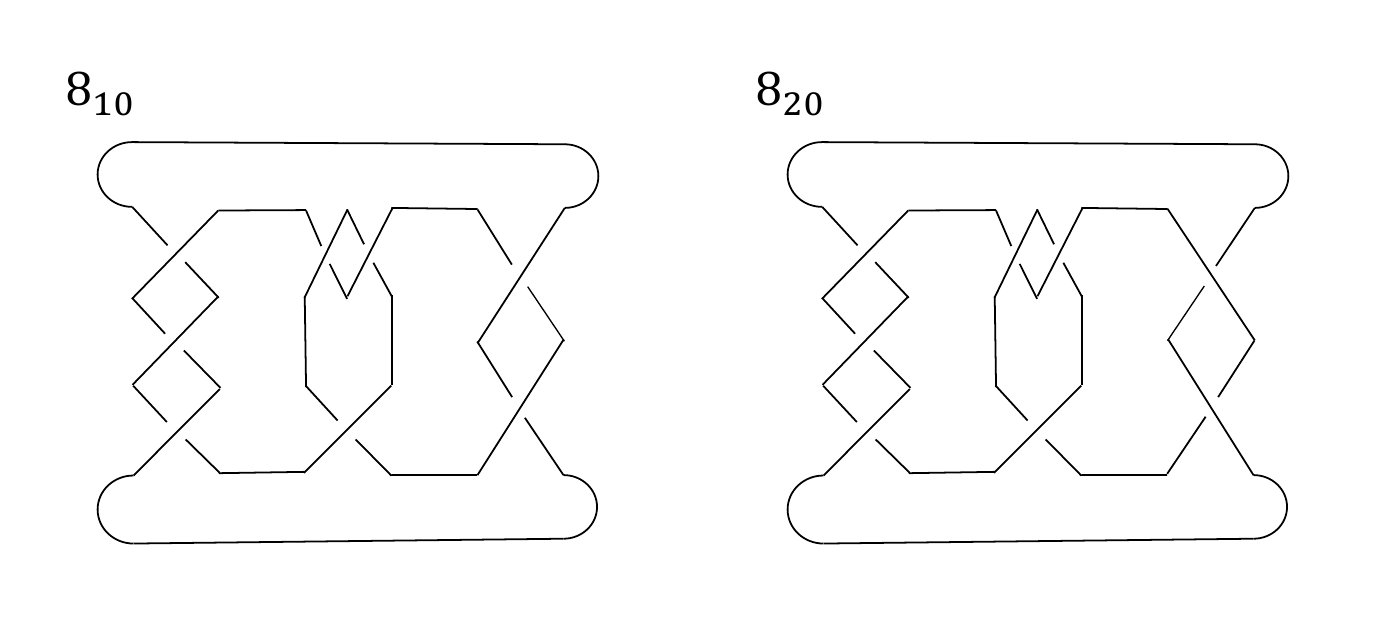}
    \caption{Diagrams of $8_{10}$ and $8_{20}$ sharing the same shadow.}
    \label{fig:810820}
\end{figure}
\begin{table}[htbp]
\centering
\caption{Values of $u^{\Delta}$ and $a_2$ for knots grouped by shadows}
\label{table2}
\begin{tabular}{c|c|c|c|c}
Pair & Knot $M$ & $u^{\Delta}(M)$ & $a_2(M)$ 
& alternating / almost alternating\\
\hline
$(8_{10}, 8_{20})$ 
& $8_{10}$ & $3$ & $3$ & alternating\\
& $8_{20}$ & $2$ & $2$ & almost alternating\\
\hline
$(8_{15}, 8_{21})$ 
& $8_{15}$ & $4$ & $4$ & alternating\\
& $8_{21}$ & $2$ & $0$ & almost alternating\\
\hline
$(9_{22}, 9_{43})$ 
& $9_{22}$ & $1$ & $-1$ & alternating\\
& $9_{43}$ & $3$ & $1$ & almost alternating\\
\hline
$(9_{25}, 9_{44})$ 
& $9_{25}$ & $2$ & $0$ & alternating\\
& $9_{44}$ & $2$ & $0$ & almost alternating\\
\hline
$(9_{30}, 9_{45})$ 
& $9_{30}$ & $1$ & $-1$ & alternating\\
& $9_{45}$ & $2$ & $2$ & almost alternating\\
\hline
$(9_{36}, 9_{42})$ 
& $9_{36}$ & $3$ & $3$ & alternating\\
& $9_{42}$ & $2$ & $-2$ & almost alternating\\
\hline
\end{tabular}
\end{table}

Moreover, the following propositions hold
(see Table \ref{table2}).

\begin{proposition}
Let $M$ and $M'$ be Montesinos knots with 
$c(M)=c(M')$ that admit minimal crossing diagrams sharing the same shadow, where this shadow realizes the minimal crossing number.  
Assume that $M$ is alternating, while $M'$ is almost alternating. Then each of the following three cases is realized.
\begin{enumerate}
 \item \( u^{\Delta}(M) > u^{\Delta}(M') \),
 \item \( u^{\Delta}(M) < u^{\Delta}(M') \),
 \item \( u^{\Delta}(M) = u^{\Delta}(M') \).
\end{enumerate}
Here, $c(M)$ denotes the minimal crossing number of $M$.
\end{proposition}

\begin{example}
The three cases in the above proposition are realized by the following examples :
\begin{enumerate}
    \item \( u^{\Delta}(8_{10}) > u^{\Delta}(8_{20}) \),
    \item \( u^{\Delta}(9_{22}) < u^{\Delta}(9_{43}) \),
    \item \( u^{\Delta}(9_{25}) = u^{\Delta}(9_{44}) \).
\end{enumerate}
\end{example}

Thus, the alternation of a knot does not determine the value of $u^{\Delta}(M)$, even among knots admitting minimal crossing diagrams with the same shadow.

\begin{proposition}
Let $M$ and $M'$ be Montesinos knots with 
$c(M)=c(M')$ that admit minimal crossing diagrams sharing the same shadow, where this shadow realizes the minimal crossing number.  
Then each of the following two cases is realized.
\begin{enumerate}
 \item \( u^{\Delta}(M) = |a_2(M)|, \quad
 u^{\Delta}(M') \neq |a_2(M')| \),
 \item \( u^{\Delta}(M) = |a_2(M)|, \quad
 u^{\Delta}(M') = |a_2(M')| \).
\end{enumerate}
\end{proposition}

\begin{example}
The two cases in the above proposition are realized by the following examples, respectively :
\begin{enumerate}
 \item \( M = 8_{15}, \quad M' = 8_{21} \),
 \item \( M = 8_{10}, \quad M' = 8_{20} \).
\end{enumerate}
\end{example}

Thus, the equality $u^{\Delta}(M) = |a_2(M)|$ is not determined by the shadow.


\section*{Acknowledgments}
The author would like to express his sincere gratitude to Professor Makoto Sakuma for his valuable advice and continuous support.

Finally, he would like to thank his family for their unwavering support and understanding throughout this work.





\end{document}